\documentclass[a4paper, 12pt]{article}
\usepackage[utf8]{inputenc}
\usepackage[T2A]{fontenc}
\usepackage{amsmath,amssymb,amsthm}
\usepackage[a4paper,hmargin=2.5cm,vmargin=2.5cm]{geometry}
\usepackage[russian,english]{babel}
\usepackage{zam}

\newtheorem{prop}{Proposition}[section]
\newtheorem{theorr}[prop]{Theorem}
\newtheorem{lemmm}[prop]{Lemma}

\newtheorem{conjecture}[prop]{Conjecture}

\theoremstyle{definition}
\newtheorem{deff}[prop]{Definition}
\newtheorem{nott}[prop]{Notation}

\theoremstyle{remark}
\newtheorem{remm}[prop]{Remark}

\pretitle{%
    \begin{center}
    \large\scshape
}

\posttitle{%
    \end{center}
    \vspace{1.2em}
}

\preauthor{%
    \begin{center}
    \large\scshape
}

\postauthor{%
    \end{center}
    \vspace{1em}
}

\predate{}
\postdate{}

\title{%
    Topological Classification of Univariate Sparse Polynomials via Tropical Geometry
}

\author{%
    Lev Vladimirov\thanks{HSE, Vladimirov.L.S@hse.ru}
}

\date{} 
\begin{document}

\maketitle

\begin{abstract}

We give the topological classification of some simplest univariate complex fewnomials, and find tropical fans of the fewnomial strata with each topological type. Specifically, for the support
$
A=\{d,d+1,d+2,d+3,d+4\},\quad d\geq3,
$
we classify the topological equivalence classes of polynomials having two critical values of ramification type $(2,2)$. There is one class for odd $d$ and two classes for even $d$. We describe the branches at infinity of the corresponding one-dimensional stratum using Newton polytopes of certain complete intersection representations of the sought strata, and determine their tropicalizations. The combinatorial types of the obtained tropical polynomials distinguishes all three kinds of topological equivalence classes. These computations suggest a correspondence between topological equivalence classes of polynomials and combinatorial types of tropicalizations of their branches at infinity, as well as a possible periodicity phenomenon under translations of the support.

\end{abstract}

\pagebreak

\tableofcontents

\pagebreak

\begin{section}{Introduction}
    This paper studies degenerations of univariate complex polynomials using the methods of tropical geometry. A natural question arising in this setting is whether the tropicalization of such a degeneration (and, in particular, the corresponding initial or central fiber) retains information about the geometry and topology of a generic polynomial in the family.

    In this paper, we study a topological version of this question. As is well known, a univariate complex polynomial defines a branched covering $P: \overline{\CC} \to \overline{\CC}$. For such coverings, one can define topological equivalence, considering two coverings equivalent if they can be intertwined by homeomorphisms of the source and the target. A natural invariant preserved under topological equivalence is the ramification data, which does not, in general, uniquely determine the topological equivalence class. In Section \ref{section_topological_classification}, we review the classical problem of topological classification of polynomials, following \cite{ZV, ZD}.

    In Section \ref{section_tropicalization}, relying on \cite{BS, KaKhE}, we recall the necessary concepts from tropical geometry and the theory of Newton polytopes. The key notion for our purposes is the tropicalization of a one-parameter family of complex polynomials. Namely, to a family 
    \[
    g_t(x) = \sum_{i=0}^n a_i(t) x^i
    \]
    one associates a tropical polynomial 
    \[
    \mathrm{Trop}(g_t)(x) = \bigoplus_{i=0}^n \mathrm{val}(a_i) \odot x^i,
    \]
    where $\mathrm{val}(a_i)$ denotes the valuation of the Laurent series $a_i(t)$, i.e., the highest exponent of $t$ with a nonzero coefficient as $t \to \infty$.

    In the remainder of the paper, we investigate the following conjecture in the simplest nontrivial cases: for a ''sufficiently generic'' family of polynomials with a given support and a given topological equivalence class of a generic polynomial in the family, the tropicalization contains enough information to distinguish this topological equivalence class from the others. This approach may provide a new method for studying the topological classification: Theorem \ref{theorem_braid_group_action} provides a clear description of the classes of topological equivalence, but does not provide a simple method for computing their number, even for polynomials of degree greater than $11$ (see, e.g., \cite{ZV}). We also expect tropicalizations to be related to the combinatorics of the Morse polytope described in \cite{EV} and \cite{Vo}.  

    In Section \ref{section_computation}, we investigate this question for the support $A = \{d, d+1, d+2, d+3, d+4\}$, which is the simplest support for which topologically distinct polynomials with the same ramification type can occur. In other words, we consider polynomials where the critical value $0$ has a critical point of local degree $d$, and two other critical values, each having two critical points of local degree $2$.

    In this paper, we take the first steps toward formulating a tropical correspondence theorem: by fixing the support and the topological type as described above, we obtain a one-dimensional stratum on which we can test our conjecture and describe the tropical correspondence in this case, enabling us to refine our hypothesis. This goal is pursued in Propositions \ref{prop_2dim}, \ref{prop_1dim_skew}, and \ref{prop_1dim_coll} and summarized in Theorem \ref{theorem_main_theorem}.

\end{section}

    \begin{section}{Topological classification}\label{section_topological_classification}

    In this section, we briefly review the background of the topological classification. For further details, see, e.g., \cite{ZV}.     
    \begin{subsection}{Problem statement}
        \begin{deff}
            Let $P_1: \overline{\CC} \to \overline{\CC}$ and $P_2: \overline{\CC} \to \overline{\CC}$ be two polynomials of degree $n$, where $\overline{\CC} \simeq S^2$ is the Riemann sphere. We say that $P_1$ and $P_2$ are {\it topologically equivalent} if there exist two orientation-preserving homeomorphisms $h_1, h_2:S^2 \to S^2$, where $h_1$ acts on the target and $h_2$ acts on the source, such that $h_1 \circ P_2 = P_1 \circ h_2$.
        \end{deff}

        The main purpose of topological classification is to compute the number of classes of topological equivalence in the set of all polynomials of degree $n$ and describe them in terms of mathematical invariants. The analogous classification problem can be posed for polynomials with prescribed support.

        \begin{deff}
            A finite subset $A \subset \ZZ$ is called {\it a support}. For a support $A$, denote by $\CC^A$ the set of polynomials whose nonzero coefficients occur only at exponents belonging to $A$.
        \end{deff}

    \end{subsection}
        
    \begin{subsection}{Braid group action}
        A complex polynomial $P$ of degree $n$ defines a branched covering $S^2 \to S^2$ of degree $n$ whose finite branch values are precisely the finite critical values of $P$. Let $y_1, \ldots, y_k$ be all the finite critical values of $P$ (note that $x =\infty$ is also a critical point, with critical value $y = \infty$). Choosing appropriate loops starting and ending at a general point $y_0$ and going around the branch values, we may associate to the polynomial $P$ a sequence of permutations $[g_1, \ldots, g_k] \in S_n$; these permutations generate the monodromy group of $P$. $g_i$ is a permutation of the preimages $P^{-1}(y_0)$ obtained by lifting the loop around $y_i$.

        Let $\lambda^{(i)} \vdash n, \lambda^{(i) = {d_{i1}, \ldots, d_{ip}}}$ be the cyclic structure of the permutation $g_i$, $i = 1, \ldots, k$, or equivalently, the multiplicity partition for the roots of the equation $P(x) = y_i$. The permutations $g_i$ must satisfy the following conditions:

        \begin{enumerate}[label=(\roman*)]
            \item\label{property_passport_cactus_1} The product $g_1g_2\ldots g_k$ is a full cycle: the points $P^{-1}(y_0)$ may be numbered in such a way that $g_1g_2\ldots g_k = (12\ldots n)$. This follows from the fact that the loop going around all the points $y_i$ may also be considered as a loop going around infinity.

            \item\label{property_passport_cactus_2} For a polynomial of degree $n$, the total ramification over the finite critical values equals $n-1$, since $P^\prime$ has degree $n-1$. Thus, the partitions satisfy $\sum\limits_{i = 1}^k\sum\limits_{j = 1}^{p_i}(d_{ij} - 1) = n - 1$.

        \end{enumerate}

        \begin{deff}
            A sequence of permutations $[g_1,\ldots,g_k]$ satisfying the conditions \ref{property_passport_cactus_1} and \ref{property_passport_cactus_2} is called {\it a (rooted) cactus}. A sequence of partitions $[\lambda^{(1)}, \ldots,\lambda^{(k)}]$ of the number $n$ satisfying the condition \ref{property_passport_cactus_2} is called {\it a polynomial passport}.
        \end{deff}

        The following theorem is one of the fundamental results underlying this study.
        \begin{theorr}
        \begin{enumerate}[label=(\roman*)]
            \item For any pairwise distinct $y_1, \ldots, y_k \in \CC$, and for an arbitrary cactus $[g_1, \ldots, g_k]$, there exists a polynomial $P$ with the critical values $y_i$ and the corresponding monodromy generators $g_i$. This polynomial is unique up to an affine change of the variable $x$.
            \item For any polynomial passport $[\lambda^{(1)}, \ldots, \lambda^{(k)}]$ there exists at least one cactus $[g_1, \ldots, g_k]$ such that $\lambda^{(i)}$ is the cyclic structure of $g_i$.   
        \end{enumerate}
        \end{theorr}

        Let an operation $\sigma_i$ ($i = 1, \dots, k-1$) act on a cactus $[g_1, \dots, g_k]$ in the following way:
    \begin{equation*}
    \sigma_i(g_i) = g_{i+1}, \quad \sigma_i(g_{i+1}) = g_{i+1}^{-1} g_i g_{i+1}, \quad \text{and} \quad \sigma_i(g_j) = g_j \text{ for } j \neq i.
    \end{equation*}
    
    A direct calculation shows that this action preserves the product $g_1 \dots g_k$. The following theorem of S.~Zdravkovska \cite{ZD} is fundamental:

    \begin{theorr}\label{theorem_braid_group_action}
    \begin{enumerate}
        \item The operations $\sigma_i$, $i = 1, \dots, k-1$ generate an action of the braid group $B_k$. That is, they satisfy the relations
    \begin{equation*}
    \sigma_i \sigma_j = \sigma_j \sigma_i \text{ for } |i - j| \ge 2, \quad \text{and} \quad \sigma_i \sigma_{i+1} \sigma_i = \sigma_{i+1} \sigma_i \sigma_{i+1}.
    \end{equation*}

        \item The topological equivalence classes of complex polynomials are in one-to-one correspondence with the orbits of this action.
    \end{enumerate}
        
    \end{theorr} 

 \end{subsection}
\end{section}

\begin{section}{Introduction to tropical degenerations}\label{section_tropicalization}

In this section, we recall some basics of tropical geometry. For details, see, e.g., \cite{BS,KaKhE}.

\begin{subsection}{Newton polytopes }
    \begin{deff}
For each monomial $a_{\alpha}x^{\alpha}$ with $a_{\alpha} \neq 0$, where $\alpha = (\alpha_1, \ldots, \alpha_n) \in \ZZ^n$, we associate the exponent vector $\alpha$. {\it The Newton polytope} of a Laurent polynomial $F$ is the convex hull of all such exponent vectors.
\end{deff}

\begin{deff}
A face $\Gamma$ of a polytope $\Delta$ is said {\it to be exposed by} $v \in (\RR^n)^*$ if

$\Delta^v = \{a \in \Delta: \langle v, a \rangle = \max\limits_{b\in\Delta} \langle v, b\rangle \}$. We denote $\max\limits_{b\in\Delta} \langle v, b\rangle $ by $\Delta(v)$. We call $v$ {\it an outward normal vector} exposing $\Gamma$.
\end{deff}

\begin{deff}
Let $\Delta_1, \Delta_2, \dots, \Delta_{k}$ be a collection of polytopes.
Faces $\Gamma_1 \subset \Delta_1, \dots, \Gamma_k \subset \Delta_k$ are called {\it compatible} if they admit a common outward normal vector $v$ and each $\Gamma_i$ is the face of $\Delta_i$ exposed by $v$.
\end{deff}

\end{subsection}

\begin{subsection}{Toric manifolds}

\begin{deff}
    An irreducible $n$-dimensional complex algebraic manifold $X$ with an action of the torus $(\CC^*)^n$ such that one of the orbits is Zariski dense in $X$ is called {\it toric}.
\end{deff}

 Let $C \subset (\CC^*)^n$ be a smooth algebraic curve. 

 \begin{deff}
{\it A smooth toric compactification} of a smooth complex algebraic curve $C \subset(\CC^*)^n$ is the closure $\overline{C}$ of $C$ in some compact toric manifold $(\CC^*)^n \subset X$ with the finite set $I:=\overline{C} \setminus C$ which is called the set of {\it points at infinity} of $C$.
\end{deff}

\begin{theorr}
     Every smooth complex algebraic curve $C \subset(\CC^*)^n$ admits a smooth toric compactification after passing to a suitable smooth complete toric variety.
\end{theorr}

\end{subsection}

\begin{subsection}{Tropical geometry}

    \begin{deff}
The {\it tropical semifield} $\mathbb{T}$ is defined as the set $\mathbb{R} \cup \{-\infty\}$ equipped with the operations
\[
a \odot b := a + b \quad \text{and} \quad a \oplus b := max(a,b) \text{.}
\]
\end{deff}

\begin{remm}
These operations reflect the behavior of degrees of polynomials under multiplication and addition. In particular, the degree of a sum cannot exceed the maximum of the degrees of its summands, and equality may fail because of cancellation.
\end{remm}

In this semifield, the additive identity is $-\infty$, while the multiplicative identity is $0$. This is also reflected in subsequent definitions. We next introduce the main geometric object used in this paper.

\begin{deff}
A {\it tropical polynomial} of degree $n$ is an expression
\[
a_0 \oplus a_1 \odot x \oplus \dots \oplus a_n \odot x^n,
\]
where  $a_i \in \mathbb{T}$ for $0 \leq i \leq n$ and $a_n \neq -\infty$ .
\end{deff}

\begin{remm}
A tropical polynomial has a formal expression analogous to an ordinary polynomial, but its operations are interpreted in the tropical semifield. 
\end{remm}

\begin{deff}
In the notation of the definition above, choose a local parameter $u$ on a small neighborhood $U$ of a point $z \in \overline{C}\setminus C$, $u(z) = 0$. Consider the parameter $t = u^{-1}$, so that natural projection $\pi: \overline{C}\setminus I \to C$ is represented by $(f_1(t), \dots, f_n(t))$, where the $f_i(t) = \sum\limits_{-\infty}^{k \leq d_i} a_{i,k} t^k$ ($a_{i,d_i} \neq 0$) are Laurent series in the variable $t$. This parametrization of the curve $C$ is called {\it a branch of} $C$ {\it at infinity}. The vector $(d_1, d_2, \dots, d_n) \in \mathbb{Z}^n$ is called {\it the valuation vector} of the branch at the point $z$. The integer $d_i$ is called {\it the valuation} of $f_i$ at $z$, and $a_{i,d_i}$ is called {\it the initial coefficient}. The valuation and the initial coefficient are denoted by $\mathrm{val}_{f_i}$ and $\mathrm{in}_{f_i}$, respectively.

\end{deff}

\begin{remm}
    We use the convention 
    \[
        \operatorname{val}_f= \max\{{k: a_k \neq 0}\} \text{,}
    \]
    corresponding to the asymptotic order of $f(t)$ as $t \to \infty$. By this convention $\operatorname{val}_0 = -\infty$.
\end{remm}
\begin{remm}
The valuation does not depend on the choice of $t$. The initial coefficient depends on the choice of $t$, but under the change of local parameter $t \mapsto st$ the initial term $a_{i,d_i}t^{d_i} \mapsto (a_{i,d_i}s^{d_i})t^{d_i}$. Thus, the projective class of the vector of initial coefficients defines a point in the quotient torus $\mathbb{C^*}\mathbb{P}^{n-1}_{(d_1, d_2, \dots, d_n)} = (C^*)^n/((x_1,x_2,\ldots, x_n)\sim(\lambda^{d_1}x_1, \ldots, \lambda^{d_n}x_{n}))$. .
\end{remm}

\begin{deff}
    A {\it polyhedral cone} in $\RR^n$ is the set of all nonnegative linear combinations of a finite collection of vectors. A {\it face} of a cone is its intersection with the boundary of the half-space which contains this cone. A cone is called {\it strongly convex} if it contains no nontrivial linear subspace. A {\it fan } is a collection of strongly convex cones such that for each cone all its faces are also contained in the fan and the intersection of any two cones is a common face of both.        
\end{deff}

For the algebraic curve, the $1$-dimensional cones corresponding to the valuation vectors at the points at infinity define a fan. This fan is a balanced $1$-fan in the following sense.

\begin{deff}
    A $\Lambda$-weighted $k$-fan is a fan $\mathcal{F} \subset \RR^n$ equipped with a weight function $c: \mathcal{F}_k \to \Lambda$ defined on the set $\mathcal{F}_k$ of $k$-dimensional cones. 
    The support $|\mathcal{F}|$ of $\mathcal{F}$ is the union of all the cones $\sigma_i \subset \RR^n$, where $\sigma_i \in \mathcal{F}_k$ and $c(\sigma_i) \neq 0$. 
    Two $\Lambda$-weighted $k$-fans $\mathcal{F}_1$ and $\mathcal{F}_2$ are equivalent if 1) $|\mathcal{F}_1| = |\mathcal{F}_2|$ 2) weight functions $c_1$ and $c_2$ induce the same function on the common subdivision of $\mathcal{F}_1$ and $\mathcal{F}_2$.
\end{deff}

\begin{deff}
    Let $\mathcal{F}$ be a $\Lambda$-weighted $k$-fan. For the cone $\sigma_i \in \mathcal{F}_k$ denote by $L^\perp_i \subset (\RR^n)^*$ the space dual to the whole space $L$ which contains $|\sigma_i| \subset\RR^n$. Let $O$ be the orientation of $|\sigma_i|$. By $e_i^\perp(O) \in \Lambda^{n-k}L_i^\perp$ we denote the vector of the dimension $n-k$ such that 1) integral volume of $|e_i^\perp(O)|$ in $L_i^\perp$ is equal to $1$; 2) the orientation of $|e_i^\perp(O)|$ is induced by the orientation $O$ of the cone $\sigma_i$ and the standard orientation of $\RR^n$. 

    A weighted $k$-fan $\mathcal{F}$ is balanced if for any orientation of the cone $|\rho|$ of the dimension $k-1$, where $\rho \in \mathcal{F}_{k-1}$, the next equation is satisfied 

    \begin{equation*}
        \sum\limits_i e_i^\perp(O(\rho)) \cdot c(\sigma_i) = 0 \text{.}
    \end{equation*}
\end{deff}

For the $1$-fan, corresponding to the algebraic curve, we choose the weight function equal to the number of the points at infinity, corresponding to the valuation vector.

\begin{theorr}\label{theorem_restriction_on_faces}

\begin{itemize}
    \item Let $C \subset (\CC^*)^n$ be a curve defined by a system of polynomial equations $F_1 = F_2 = \dots = F_k = 0$. For each $i \in \{1, \dots, k\}$, denote by $\Delta_i$ the Newton polytope of the $i$-th equation. Then for every point $z$ at infinity of the curve $C$ and every branch at infinity at this point, there exist faces $\Gamma_1 \subset \Delta_1, \dots, \Gamma_k \subset \Delta_k$ such that the valuation vector of this branch at $z$ exposes the face $\Gamma_i$ of the polytope $\Delta_i$ for all $i \in \{1, \dots, k\}$, and the vector of initial coefficients satisfies the corresponding initial-form equations $F_1|_{\Gamma_1} = \dots = F_k|_{\Gamma_k} = 0$.

    \item For $n - 1$ equations $F_1|_{\Gamma_1} = \dots = F_{n-1}|_{\Gamma_{n-1}} = 0$  on $(\CC^*)^n$, if the initial equations meet properly
    and the common zero of the initial forms has Jacobian rank $n - 1$ transverse to the $(\CC^*)$ weight orbit of $\CC^*\PP^{n-1}_{(d_1, \ldots, d_n)}$, then the implicit-function argument produces a branch with exact primitive valuation vector exposing $\Gamma_i$.
\end{itemize}
    
\end{theorr}

\begin{deff}
    Let 
    \[
    g_t(x) = \sum_{(i_1, i_2, \dots, i_n) \in \ZZ_{\ge 0}^n} a_{i_1, i_2, \dots, i_n}(t) x_1^{i_1} \dots x_n^{i_n}
    \]
    be a one-parameter family of polynomials whose coefficients $a_{i_1, i_2, \dots, i_n}(t)$ are Laurent series. {\it The tropicalization} of $g_t(x)$ is defined as the tropical polynomial
    \[
    \mathrm{Trop}(g_t) = \bigoplus_{(i_1, i_2, \dots, i_n)\in\ZZ_{\ge0}^n} \mathrm{val}\! \left(a_{i_1, i_2, \dots, i_n}(t)\right) \odot x_1^{i_1} \odot \dots \odot x_n^{i_n}.
    \]
\end{deff}

Tropicalization is a useful tool, as it allows one to reduce certain algebraic questions (for instance, in enumerative geometry) to problems concerning polynomials over the tropical semifield.

\end{subsection}
\end{section}

\begin{section}{Computations}\label{section_computation}

\begin{subsection}{Topological equivalence classes}
    In this section, we consider the curve $C \subset\CC^A$ consisting of polynomials $f\in C^{A}$ for $A = \{d, d+1,d+2,d+3,d+4\}$, where $d \ge 3$ is an integer
and each polynomial has two nonzero critical values, each having two critical points of local degree $2$. 

By Theorem \ref{theorem_braid_group_action}, the topological equivalence classes correspond to the braid-group orbits of the corresponding cacti. A direct computation of such orbits for polynomials with the passport $[[d,1,1,1,1],[2,2,1,\ldots,1],[2,2,1,\ldots,1]]$ yields the following proposition.

\begin{prop}\label{proposition_number_of_topological_classes}
    Let $d\ge3$. Then the set of polynomials of degree $d+4$ with passport $[[d,1,1,1,1],[2,2,1,\ldots,1],[2,2,1,\ldots,1]]$ consists of two topological equivalence classes for even $d$ and of a single class for odd $d$.  
\end{prop}

\begin{proof}
Each cactus is equipped with an ordering of the permutations $[g_1,g_2,g_3]$. It is also characterized by the common nonstable elements of the permutations $g_1,g_2,g_3$. The complete list of the structures of the triple $\{g_1,g_2,g_3\}$ of permutations is as follows:

\bigskip
\quad One single connection with the cycle:
\begin{enumerate}
   
\item\label{1_cycle_con_1}  $\{(bc)(de),(ab)(cd),(\ldots e \ldots)\}$ \\ \\ One double connection with the cycle:
\item\label{1_cycle_double_con_1_3}  $\{(bc)(de),(ab)(cd),(\ldots d \ldots)\}$
\item\label{1_cycle_double_con_2_2}  $\{(ab)(cd),(bc)(de),(\ldots c \ldots)\}$ \\ \\ Two single connections with the cycle:
\item\label{2_cycle_singe_con_1_3} $\{(bc)(ef),(ab)(cd),(\ldots d \ldots f)\}$
\item\label{2_cycle_singe_con_2_2_diff}  $\{(ab)(ef),(bc)(de),(\ldots c \ldots f)\}$
\item\label{2_cycle_singe_con_2_2_sim}  $\{(ab)(de),(bc)(ef),(\ldots c \ldots f)\}$\\ \\ Two connections with the cycle: one single and one double 
\item\label{2_cycle_singe_double_con_1_3}  $\{(bc)(ef),(ab)(cd),(\ldots c \ldots f)\}$
\item\label{2_cycle_singe_double_con_2_2}  $\{(bc)(df),(ab)(ef),(\ldots c \ldots f)\}$\\ \\ Two double connections with the cycle:
\item\label{2_cycle_double_con_2_2}  $\{(ac)(df),(bc)(ef),(\ldots c \ldots f)\}$\\ \\ Three simple connections with the cycle:
\item\label{3_cycle_simple_case1}  $\{(bc)(de),(ab)(fg),(\ldots c \ldots e \ldots
g)\}$
\item\label{3_cycle_simple_case2}  $\{(bc)(de),(ab)(fg),(\ldots c \ldots g \ldots
e)\}$\\ \\ Three connections with the cycle: two single one double:
\item\label{3_cycle_simple_double_case1}  $\{(ac)(de),(bc)(fg),(\ldots c \ldots e \ldots
g)\}$
\item\label{3_cycle_simple_double_case2}  $\{(ac)(de),(bc)(fg),(\ldots c \ldots g \ldots
e)\}$\\ \\ Four simple connections with the cycle:
\item\label{4_cycle_simpe_case1}  $\{(ab)(cd),(ef)(gh),(\ldots a \ldots c\ldots e \ldots g)\}$
\item\label{4_cycle_simpe_case2}  $\{(ab)(cd),(ef)(gh),(\ldots a \ldots e\ldots c \ldots g)\}$
\end{enumerate}

For each cactus, we describe below how to obtain the structure \ref{1_cycle_con_1} and the order $[g_3,g_2,g_1]$ for odd $d$ by applying the action of the braid group. For even $d$, the same argument applies to one of the classes. For the second class we show how to obtain the structure \ref{2_cycle_singe_con_2_2_sim} and the order $[g_3,g_1,g_2]$. The proof proceeds as follows: we start with the class \ref{1_cycle_con_1} and the order $[g_3,g_1,g_2]$ and then for each other structure and order we show how to obtain either the initial case or another case that has already been considered.

\begin{itemize} 
\item Here we prove that all cacti with the structure \ref{1_cycle_con_1} are equivalent to some cactus with the same structure and the order $[g_3, g_2, g_1]$. We move $g_3$ to the left preserving the structure and then, if necessary, move $g_2$ to the left preserving the structure to obtain the order $[g_3,g_2,g_1]$. After relabelling all these cacti are equivalent to $$[(1 2 \ldots (d-1) (d+4)),((d+2) (d+1))((d+3) d),((d+1) (d+3))(d (d+4))]\text{.}$$

\item For the cacti with the structure \ref{1_cycle_double_con_1_3} or \ref{2_cycle_singe_con_1_3}, we move $g_3$ to the left to obtain one of the structures \ref{1_cycle_double_con_1_3} or \ref{2_cycle_singe_con_1_3} and the order $[g_3, x, y]$. 

For the structure \ref{2_cycle_singe_con_1_3} and the order $[g_3,g_2,g_1]$, we swap $g_1$ and $g_2$ to obtain the order $[g_3, g_1, g_2]$ and preserve the structure. 

For the structure \ref{2_cycle_singe_con_1_3} and the order $[g_3,g_1,g_2]$, it is possible to obtain the structure \ref{1_cycle_double_con_1_3} and the order $[g_3,g_1, g_2]$. A double swap of the first two elements moves the element $f$ in the cycle $g_3$ right by one. We repeat these moves until $f$ and $d$ coincide, thereby obtaining the structure  \ref{1_cycle_double_con_1_3}. 

For the structure \ref{1_cycle_double_con_1_3} and the orders $[g_3,g_2,g_1]$ and $[g_3,g_1,g_2]$, we swap the last two and the first two elements respectively to obtain the structure \ref{1_cycle_con_1}.

\item For the structure \ref{1_cycle_double_con_2_2}, if $g_3$ is not the middle element of the order, we swap $g_1$ and $g_2$ to obtain the structure \ref{1_cycle_double_con_1_3}. 

If $g_3$ is the middle element of the order, we swap it with the right element to obtain the structure \ref{1_cycle_double_con_1_3}.

\item For the structure \ref{2_cycle_singe_double_con_1_3}, if $g_3$ is not the middle element of the order, then we may obtain the structure \ref{1_cycle_double_con_1_3} by a single or double swap of $g_1$ and $g_2$. The swap $[g_2, g_1] \to [g_1,g_2]$ preserves the structure and the swap $[g_1, g_2] \to [g_2,g_1]$ gives the structure \ref{2_cycle_singe_con_1_3}.

If $g_3$ is the middle element of the order, we swap it with the right element. The swap $[g_3,g_1] \to [g_1,g_3]$ preserves the structure and makes $g_3$ not the middle element; the swap $[g_3,g_2] \to [g_2,g_3]$ obtains the structure  \ref{1_cycle_double_con_1_3}.

 \item Here we describe how to obtain the structure \ref{2_cycle_singe_double_con_1_3}  or \ref{1_cycle_double_con_2_2} out of \ref{2_cycle_singe_con_2_2_diff} or \ref{2_cycle_singe_double_con_2_2}.

 For the structure \ref{2_cycle_singe_con_2_2_diff}, we move $g_3$ to the left to obtain the  same structure and the order $[g_3,x,y]$ or to obtain the structure \ref{1_cycle_double_con_2_2}. For the structure \ref{2_cycle_singe_con_2_2_diff}  and the order $[g_3,x,y]$, if $x = g_2$, we swap the first two elements to obtain the structure \ref{2_cycle_singe_double_con_2_2}; if $x = g_1$, we swap the last two elements for this purpose. 

 For the structure \ref{2_cycle_singe_double_con_2_2} with the order $[g_1,g_2,g_3]$, it is possible to obtain the structure \ref{2_cycle_singe_double_con_1_3}. A double swap of the last two elements gives the structure \ref{3_cycle_simple_case1} with the order $[g_1,g_2,g_3]$. For the structure \ref{3_cycle_simple_case1} and the order $[g_1,g_2,g_3]$, the double swap of the last two elements moves the point $g$ right by one in the cycle $g_3$. We repeat these moves until $g$ and $c$ coincide, thereby obtaining the structure \ref{2_cycle_singe_double_con_1_3}.
 
 For the structure \ref{2_cycle_singe_double_con_2_2}, we move $g_1$ to the left preserving the structure. If the order is $[g_1, g_2, g_3]$, then we know how to obtain the structure \ref{2_cycle_singe_double_con_1_3}. If the order is $[g_1, g_3, g_2]$, we swap the last two elements to obtain the structure \ref{2_cycle_singe_con_2_2_diff} with the order $[g_1, g_2, g_3]$ and then swap the first two elements to obtain the structure \ref{2_cycle_singe_double_con_2_2} with the order $[g_2,g_1,g_3]$. Then move $g_1$ to the left to obtain the order $[g_1, g_2, g_3]$ and preserve the structure \ref{2_cycle_singe_double_con_2_2}.

 \item For the structures \ref{3_cycle_simple_case1}  and 
\ref{3_cycle_simple_case2}, we move the element $g_2$ to the left preserving the structure and then, if necessary, we
move $g_3$ to the left to preserve the structure and obtain the order $[g_2, g_3, g_1]$ or to obtain one of the structures \ref{2_cycle_singe_double_con_1_3} and \ref{2_cycle_singe_double_con_2_2}. A double swap of the first two elements moves the element $g$ in the cycle $g_3$ right by one. We repeat these swaps until $g$ coincides with $c$ or $e$, thereby obtaining the structure \ref{2_cycle_singe_double_con_1_3} or \ref{2_cycle_singe_double_con_2_2}.

\item For the structures \ref{3_cycle_simple_double_case1} and  \ref{3_cycle_simple_double_case2},
 if $g_3$ is not the middle element of the order, we swap $g_1$ and $g_2$ to obtain the structure \ref{3_cycle_simple_case1} or \ref{3_cycle_simple_case2}.
If $g_3$ is the middle element of the order, we swap it with the right element to obtain the structure \ref{3_cycle_simple_case1} or \ref{3_cycle_simple_case2}.

\item For the structures \ref{4_cycle_simpe_case1} and  \ref{4_cycle_simpe_case2}, if at least one of $(a,c)$ and $(e,g)$ does not split $g_3$ into two equal parts,
it is possible to move one of them to obtain the coincidence of exactly two points
and consequently the structure \ref{3_cycle_simple_double_case1} or \ref{3_cycle_simple_double_case2} (we note that when $d+4$ is odd, it is always possible,
whereas when $d+4$ is even, it is not always possible because of the symmetry. This symmetry provides the second orbit in the even case).

The swap $[g_3, g_i] \to [g_i, g_3]$ preserves the distances between $a,c,e,g$ and $[g_i, g_3] \to [g_3, g_i]$ moves $(a,c)$ or $(e,g)$ right by one in the cycle $g_3$.
If $g_3$ is the middle element of the order, then, depending on which pair $(a,c)$ or $(e,g)$ we need to move, we may swap $g_3$ with the corresponding permutation.

Since $g_1g_2 = g_2g_1$, if $g_3$ is not in the middle of the order, then, depending on which pair $(a,c)$ or $(e,g)$ we need to move, we can place the corresponding permutation $g_1$ or $g_2$ respectively in the middle.

For even $d+4$, if $(a,c)$ and $(e,g)$ split $g_3$ into two equal parts, it is possible to achieve the coincidence of two pairs of vertices in the same way as described above to obtain the structure \ref{2_cycle_double_con_2_2}.

\item For the structure \ref{2_cycle_double_con_2_2} and the orders $[x,y,g_3]$ and $[y, g_3,x]$, we swap $g_3$ with $y$ to obtain the structure \ref{4_cycle_simpe_case1} or  \ref{4_cycle_simpe_case2} or \ref{3_cycle_simple_double_case1}. For the other orders, we move $g_3$ to the right by $1$ to obtain the structure \ref{2_cycle_singe_con_2_2_sim}.  

For the structure \ref{2_cycle_singe_con_2_2_sim} we move $g_3$ to the left preserving the structure and then, if necessary, move $g_1$ to the left to obtain the order $[g_3,g_1,g_2]$. Then we swap $g_3$ and $g_1$ to preserve the structure and 
obtain the order $[g_1,g_3,g_2]$. Then we swap $g_3$ and $g_2$ to obtain
the structure \ref{2_cycle_double_con_2_2} with the order $[g_1,g_2,g_3]$.

\item For even $d+4 = 2k+4$, we should also note that for the structure \ref{2_cycle_double_con_2_2} and the order $[g_3,x,y]$, it is also possible to obtain the structure \ref{2_cycle_singe_con_2_2_sim} by swapping $x$ and $y$. The symmetric cacti with the structures \ref{4_cycle_simpe_case2}, \ref{2_cycle_double_con_2_2}, \ref{2_cycle_singe_con_2_2_sim} are equivalent to the symmetrical cactus with the structure \ref{2_cycle_singe_con_2_2_sim} and the order $[g_3,g_1,g_2]$, namely $$[( 23\ldots (k)(k+3)\ldots(2k+3)(2)), ((k+1)(k+2))((2k+4)(1)) , ((k+1) (k+3))((2k+4)(2))]$$.
\end{itemize}

\end{proof}

\end{subsection}

\begin{subsection}{Lyashko-Looijenga map}

    For each support $A$, each polynomial passport defines a stratum $A_X \subset \CC^A$ consisting of polynomials with the passport $X$. If $A_X$ consists of more than one topological equivalence class, then $A_X$ is not connected. Moreover, the following proposition is well-known (see, e.g., \cite{GJV}).

    \begin{prop}\label{proposition_connected_components}
        The topological equivalence classes of meromorphic functions $f:S^2\to S^2$ are in one-to-one correspondence with the connected components of the space of functions having the same passport.
    \end{prop}

    In order to describe these connected components for $C$, we introduce the Lyashko-Looijenga map for polynomials with support $A$ and $k$ distinct critical values as
    
    \begin{align}
            & \operatorname{LL}: \CC^A \to \CC^k\\
            &f \mapsto g(t)= \prod\limits_{i = 1}^k(t-t_i)\text{,}
    \end{align}

    where $t_i$ is a critical value of $f$.

    The following theorem is proved in \cite{LZV}.

    \begin{nott}
         To a polynomial $f(x)=x^{n+1}+\ldots+a_n$ with $k$ critical values we associate the collection of partitions $X = [X_1,X_2,\ldots, X_k] = [[1^{k_{1,1}}\ldots n^{k_{1,n}}],[1^{k_{2,1}}\ldots n^{k_{2,n}}],\ldots,[1^{k_{k,1}},\ldots,n^{k_{k,n}}]]$, where $k_{i,j}$ is the number of roots of $f^\prime$ with multiplicity $j$ corresponding to the $i$-th critical value of $f$. We refer to $X$ as the critical-point passport of $f$. 

        The number $l(X_i) = \sum\limits_{j=1}^{n}k_{i,j}$ is called the {\it length } of the $i$-th critical value.

        The number $A(X_i) = \sum\limits_{j=1}^{n}j\cdot k_{i,j}$ is called the {\it degeneracy } of the $i$-th critical value.
        
        To $X = [X_1,X_2,\ldots, X_k]$ we associate the partition $T=[1^{t_1}\ldots n^{t_n}]$ of $n$, where $t_i = |\{X_j \in X: A(X_j) = i\}|$.

        $\operatorname{\#AUT}(X_i) = k_{i,1}!\ldots k_{i,n}!$ is the number of automorphisms of $X_i$ preserving the multiplicities of critical points.

        $\operatorname{\#AUT}(T) = t_1!\ldots t_n!$ is the number of automorphisms of $X$ preserving the degeneracy.

        $\operatorname{\#AUT}(X)$ is the number of automorphisms of $X$ preserving the ramification type of critical points.

    \end{nott}
        
    \begin{theorr}\label{theorem_degree_of_lyashko_looijenga_map}
        The restriction of the Lyashko–Looijenga map to the stratum $X$ of univariate complex monic polynomials of degree $n+1$ with $k$ critical values is an unbranched covering. Its degree is equal to

        $$(n+1)^{k-1}\frac{\operatorname{\#AUT}(T) }{\operatorname{\#AUT}(X)}\prod\limits_{i = 1}^k \frac{(n - A(X_i) )!}{(n+1-A(X_i) - l(X_i))!\cdot\operatorname{\#AUT}(X_i)}$$. 
    \end{theorr}

    Recall that we consider the curve $C \subset\CC^A$ consisting of polynomials $f\in C^{A}$ for $A = \{d, d+1,d+2,d+3,d+4\}$, where $d \ge 3$ is an integer
    and each polynomial has two nonzero critical values, each having two critical points of local degree $2$. 

    \begin{prop}\label{proposition_description_of_C1}
    Let $d\ge3$. Then the stratum of polynomials of degree $d+4$ with the passport $[[d,1,1,1,1],[2,2,1,\ldots,1],[2,2,1,\ldots,1]]$ is smooth. Moreover for even $d$, one of the two topological equivalence classes of such polynomials consists precisely of those polynomials that, after an affine normalization, satisfy $f^\prime(x) = cx^{d-1}(x^2-x_1^2)(x^2-x_2^2)$. 
    \end{prop}
    
    \begin{proof}
         By Theorem \ref{theorem_degree_of_lyashko_looijenga_map}, the stratum of monic polynomials of degree $d+4$ with the critical-point passport $[[d-1],[1^2],[1^2]]$  (or, equivalently, with the passport \\$[[d,1,1,1,1],[2,2,1,\ldots,1],[2,2,1,\ldots,1]]$) is smooth. Hence, the original one is smooth. 
        
        After an affine change in the source and a rescaling in the image each polynomial with the passport $[[d,1,1,1,1],[2,2,1,\ldots,1],[2,2,1,\ldots,1]]$ may be rewritten in such a way that $f^\prime = (d+4)x^{d-1}(x-x_1)(x-x_2)(x-x_3)(x-x_4)$, where $f(x_1) = f(x_2) = t_1$, $f(x_3) = f(x_4) = t_2$, $f(x_1) \neq f(x_3)$, and $f(x_i) \neq 0$. Denote this set of the polynomials by $B$. Note that $B$ has the same number of topological equivalence classes and connected components as the original set; hence for $B$ both of these numbers equal $2$ if $d$ is even. Consider the Lyashko-Looijenga map $\operatorname{LL}_B$ given by $f \mapsto (t-t_1)(t-t_2)$. Then consider the set $B_1$ of polynomials such that $x_1 = -x_2$, $x_3 = -x_4$. These polynomials have the form $x^{d+4} - \frac{(d+4)(x_1^2+x_3^2)}{(d+2)}x^{d+2} + \frac{(d+4)x_1^2x_3^2}{d}x^d = g(x^2)$. $x_1^2$ and $x_3^2$ are the simple critical points of $g$. By Theorem \ref{theorem_degree_of_lyashko_looijenga_map}, the Lyashko-Looijenga map for monic polynomials of degree $\frac{(d+4)}{2}$ with the critical-point passport $[[1^1],[1^1],[\frac{d}{2}-1]]$ is a covering. Therefore, the restriction $\operatorname{LL_{B_1}}$ of $\operatorname{LL}_{B}$ to the locus $f(x) = g(x^2)$ is a covering. Thus, the set of $f$ satisfying $f(x) = g(x^2)$ is invariant under the monodromy group of the map $\operatorname{LL}_{B}$. Thus, this locus is a union of connected components of $B$. Since $B_1$ is a proper nonempty union of connected components of $B$, and $B$ has exactly two connected components, $B_1$ coincides with one of them. This proves the second part of the proposition.
        
    \end{proof}

Propositions \ref{proposition_number_of_topological_classes} and \ref{proposition_description_of_C1} may be applied to $C$ as follows.

\begin{prop}\label{proposition_topological_classes}
    Let $d \ge 3$ and let $A = \{d, d+1,d+2,d+3,d+4\}$. For $f \in \CC^A$ with two nonzero critical values, each having two critical points of local degree $2$, the topological equivalence classes are as follows:

    \begin{itemize}
        \item One class for odd $d$. Denote this class by $C_0$.
        \item Two classes for even $d$. One of the two classes consists of polynomials $f$ such that $f^\prime(x) = cx^{d-1}(x^2-x_1^2)(x^2-x_2^2)$. Denote this class by $C_1$ and the other one by $C_2$.
    \end{itemize}

    Moreover $C$ is smooth.
\end{prop}

\begin{proof}
    Indeed, after affine changes of the source and target, every polynomial with the passport $[[d,1,1,1,1],[2,2,1,\ldots,1],[2,2,1,\ldots,1]]$ can be normalized to a polynomial in $\CC^A$. This normalization identifies $C$ with a transverse affine slice of the smooth stratum considered in Proposition \ref{proposition_description_of_C1}. Therefore, $C$ is smooth. The same normalization shows that the topological equivalence classes are those described in Propositions \ref{proposition_number_of_topological_classes} and \ref{proposition_description_of_C1}.
\end{proof}

We now compute the tropicalizations corresponding to these topological equivalence classes in order to show how to distinguish them in terms of tropical geometry. The results are summarized in Propositions \ref{prop_2dim}, \ref{prop_1dim_skew}, and \ref{prop_1dim_coll}.
\end{subsection}

\begin{subsection}{Computation of Newton polytopes}\label{section_newton_polytopes_computation}
After rescaling the source and target coordinates, every $f\in C$ can be normalized so that
\begin{equation}\label{system_definition_of_polynomials} 
    \begin{cases}   
    f^\prime(x) = x^{d-1}(x-x_1)(x-x_2)(x-x_3)(x-x_4) \text{,}  \\
     x_1x_2x_3x_4 = 1 \text{,}  \\
     f(x_1) = f(x_2)\text{,}  \\
     f(x_3) = f(x_4) \text{.}
    \end{cases}
\end{equation}
    
    Then
    \begin{multline*}
        f(x) = \frac{x^{d+4}}{d+4} - \frac{x^{d+3}}{d+3}\left(x_1+x_2+x_3+x_4\right) + \frac{x^{d+2}}{d+2}\left( x_1x_2+x_1x_3+x_1x_4+x_2x_3+x_2x_4+x_3x_4\right) - \\ - \frac{x^{d+1}}{d+1}\left( x_1x_2x_3+x_2x_3x_4+x_3x_4x_1+x_4x_1x_2\right) + \frac{x^d}{d}\left(x_1x_2x_3x_4\right) \text{.}
    \end{multline*}

    Define $C^*$ as the set of polynomials in $f\in\CC^A$ such that \eqref{system_definition_of_polynomials} holds for $x_1\neq x_2$ and $x_3 \neq x_4$. Since $C$ is smooth, $C^*$ is smooth at all the points, corresponding to the polynomials from $C$. 

    \begin{equation}\label{equation_fx1equalfx2}
    \begin{split}
        f(x_1) - f(x_2) = \left(x_1-x_2\right)^3 \left( - \frac{\sum\limits_{i = 0}^{d+1} x_1^ix_2^{d-i+1}(i+1)(d-i+2)}{(d+4)(d+3)(d+2)} + \right. \\ \left. + (x_3+x_4) \cdot \frac{\sum\limits_{i = 0}^{d} x_1^ix_2^{d-i}(i+1)(d-i+1)}{(d+3)(d+2)(d+1)} - x_3x_4 \cdot \frac{\sum\limits_{i = 0}^{d-1} x_1^ix_2^{d-i-1}(i+1)(d-i)}{(d+2)(d+1)(d)} \right)  \text{,}
    \end{split}
    \end{equation}
    and 

    \begin{equation}\label{equation_fx3equalfx4}
    \begin{split}
        f(x_3) - f(x_4) = \left(x_3-x_4\right)^3 \left( - \frac{\sum\limits_{i = 0}^{d+1} x_3^ix_4^{d-i+1}(i+1)(d-i+2)}{(d+4)(d+3)(d+2)} + \right. \\ \left. + (x_1+x_2) \cdot \frac{\sum\limits_{i = 0}^{d} x_3^ix_4^{d-i}(i+1)(d-i+1)}{(d+3)(d+2)(d+1)} - x_1x_2 \cdot \frac{\sum\limits_{i = 0}^{d-1} x_3^ix_4^{d-i-1}(i+1)(d-i)}{(d+2)(d+1)(d)} \right) \text{.}
    \end{split}
    \end{equation}
        Since $x_1 \neq x_2$ and $x_3 \neq x_4$, we divide by $(x_1 - x_2)^3$ and $(x_3 - x_4)^3$. We now describe the Newton polytope $\Delta_1$ for the equation \eqref{equation_fx1equalfx2}, divided by $(x_1-x_2)^3$. The variables $x_1, x_2, x_3, x_4$ are associated with the vectors $e_1 = (1,0,0), e_2 = (0, 1, 0), e_3 = (0, 0, 1), e_4 = (-1, -1, -1)$ in $\ZZ^3$, respectively; the last one is obtained from $x_1x_2x_3x_4 = 1$. 
    \begin{nott}
        Each permutation $\sigma$ of the set $\{1,2,3,4\}$ induces a linear map $\overline{\sigma}:\ZZ^3 \to \ZZ^3$. The vectors $e_1, e_2, e_3, e_4$ satisfy $e_1 + e_2 + e_3 + e_4 = 0$. Hence, every permutation of the four vectors induces a linear automorphism of $\ZZ^3$ via $e_i \mapsto e_{\sigma(i)}$.
    \end{nott}

    $\Delta_1$ has eight vertices:

    \begin{itemize}
        \item $v_1 = (d+1, 0, 0), v_2 = (d, 0, 1), v_3 = (d - 2, -1, 0), v_4 = (d - 1, -1, -1)$. These four vertices form a parallelogram face $\text{conv}(v_1,v_2,v_3,v_4)$.
        \item $(v_5, v_6, v_7, v_8 )= (\overline{\sigma}(v_1), \overline{\sigma}(v_2), \overline{\sigma}(v_3), \overline{\sigma}(v_4))$ for $\sigma = (12)$ form the face obtained from the previous one by applying the symmetry.
    \end{itemize}

    $\Delta_1$ has twelve edges:

    \begin{itemize}
        \item $[v_1,v_2], [v_2, v_3], [v_3, v_4], [v_4, v_1]$ form a parallelogram.
        \item $[v_5,v_6], [v_6, v_7], [v_7, v_8], [v_8, v_5]$ form the face obtained from the previous one by applying the symmetry.
        \item $[v_1,v_5], [v_2, v_6], [v_3, v_7], [v_4, v_8]$ are edges, parallel to $(-1,1,0)$.
    \end{itemize}
         
    $\Delta_1$ has six faces:

    \begin{itemize}
        \item $\text{conv}(v_1, v_2, v_3, v_4), \text{conv}(v_5, v_6, v_7, v_8)$ -- parallelograms.
        \item $\text{conv}(v_1, v_2, v_6, v_5)$, $\text{conv}(v_2, v_3, v_7, v_6)$, $\text{conv}(v_3, v_4, v_8, v_7)$, $\text{conv}(v_4, v_1, v_5, v_8)$ -- isosceles trapezoids. 
    \end{itemize}
        
    Moreover, the Newton polytope $\Delta_2$ for the equation \eqref{equation_fx3equalfx4}, divided by $(x_3-x_4)^3$, is the image $\overline{\sigma}(\Delta_1)$, where $\sigma = (13)(24)$.
\end{subsection}    

\begin{subsection}{Compatible faces}
    Both polytopes $\Delta_1$ and $\Delta_2$ are contained in the tetrahedron $S$, whose vertices are $(d+1, 0, 0), (0, d+1, 0), (0, 0, d+1), (-(d+1), - (d+1), -(d+1))$. For each vector $v$, consider its supporting face $F_v$ of $S$.

    There are four cases to consider:

    \begin{itemize}
        \item If $F_v$ is a vertex, then for one of $\Delta_1$ and $\Delta_2$, the face exposed by $v$ is also a vertex.
        \item If $F_v$ is two-dimensional, then for both $\Delta_1$ and $\Delta_2$ the face exposed by $v$ is also two-dimensional and each face is equal to $\Delta_i \cap F_v$; moreover, these faces are compatible.
        \item If $F_v$ is one-dimensional and it is not one of the edges $[(d+1, 0, 0), (0, d+1, 0)]$ and $[(0, 0, d+1), (-(d+1), - (d+1), -(d+1))]$, then for both $\Delta_1$ and $\Delta_2$ the face exposed by $v$ is also one-dimensional and each face is equal to $\Delta_i \cap F_v$; moreover, these faces are compatible.
        \item If $F_v$ is the edge $[(d+1, 0, 0), (0, d+1, 0)]$, then the face exposed by $v$ is also one-dimensional for $\Delta_1$ and equals $F_v$. The face of $\Delta_2$ exposed by $v$ is either 0-dimensional or equals $[(1, 1, d-1), (-(d-1) + 1, - (d-1) + 1, -(d-1))]$. Conversely, if $F_v$ is the edge $[(0, 0, d+1), (-(d+1), - (d+1), -(d+1))] \subset \Delta_2$, then the one-dimensional compatible face is equal to $[(d-2, -1, 0), (-1, d-2, 0)] \subset \Delta_1$.
    \end{itemize}

    The above discussion can be summarized as follows.
    \begin{prop}\label{proposition_compatible_faces_description}
        Using the notation introduced in Section \ref{section_newton_polytopes_computation}, the following nontrivial compatible pairs of faces of $\Delta_1, \Delta_2$ occur:

        \begin{enumerate}
            \item Four two-dimensional pairs. One of them is $\Gamma_1=\operatorname{conv}(v_1,v_2,v_6,v_5) \subset \Delta_1$ and $\Gamma_2=\overline{\sigma}(\operatorname{conv}(v_1,v_2,v_3,v_4))\subset \Delta_2$, where $\sigma = (13)(24)$ and the outward normal vector is $(1,1,1)$. The remaining pairs are obtained by applying $\overline{\sigma_i}$ to the pair above, where $\sigma_i = {(i4)}$ for $i = 1, 2, 3$.
            \item Four one-dimensional collinear pairs. One of them is $\Gamma_1 = [v_1,v_2]\subset \Delta_1$ and $\Gamma_2 =\overline{\sigma}([v_1,v_2]) \subset \Delta_2$, where $\sigma = (13)(24)$. The remaining pairs are obtained by applying $\overline{\sigma^\prime}$ to the pair above, where $\sigma^\prime \in \{(12), (34), (12)(34)\}$.
            \item Two one-dimensional skew pairs. One of them is $\Gamma_1 =[v_1,v_5] \subset \Delta_1$ and $\Gamma_2 =  \overline{\sigma}([v_3, v_7])\subset\Delta_2$, where $\sigma = (13)(24)$ and the outward normal vector is $(1,1,-1)$. The other pair is obtained by applying $\overline{\sigma}$, where $\sigma = (13)(24)$, to the pair above.
        \end{enumerate}
    \end{prop}

    \begin{proof}
        Indeed, above we considered all the possibilities of $F_v \subset S$, exposed by some vector $v$. If $\dim_{F_v} = 0$ or $2$, or if $F_v$ is one of the four nonexceptional edges we have either explicitly identified the corresponding compatible faces or shown that one of them is zero-dimensional. In the case $F_v = [v_1,v_5]$, one may consider the projection $\pi:\ZZ^3 \to \ZZ^2, (x,y,z) \mapsto (x+y,z)$. Under this projection $\pi(\Delta_1) = \operatorname{conv}((d+1,0),(d,1),(d-3,0),(d-2,-1))$ and $\pi(\Delta_2) = \operatorname{conv}((-2(d+1),-(d+1)),(-2(d-2),-(d-1)),(2,d-1),(0,d+1))$. Since $\pi(v)$ exposes the vertex $(d+1,0) \in \pi(\Delta_1)$, we have $\langle\pi(v), f_1\rangle < 0$ and $\langle \pi(v), f_2 \rangle < 0$, where $f_1 = (d,1) - (d+1,0) = (-1, 1)$ and $f_2 = (d-2,-1)-(d+1,0) = (-3,-1)$ are the vectors corresponding to the edges of $\pi(\Delta_1)$ adjacent to the vertex $(d+1,0) = \pi(v_1)=\pi(v_5)$. Hence, $\langle\pi(v), f_3\rangle < 0$ and $\langle \pi(v), f_4 \rangle < 0$, where $f_3 = (-2(d+1),-(d+1)) - (-2(d-2),-(d-1)) = (-6,-2)$ and $f_4 = (0,d+1) - (2,d-1)  = (-2,2)$ are the vectors corresponding to the edges of $\pi(\Delta_2)$ adjacent to the vertices $(0,d+1) =  \pi( \overline{\sigma}(v_3))$ and $(-2(d-2),-(d-1)) = \pi( \overline{\sigma}(v_7))$, where $\sigma = (13)(24)$. Two other edges of $\pi(\Delta_2)$ are parallel to $(1,1)$ and $\langle\pi(v),(1,-1)\rangle = -\langle\pi(v), f_1\rangle > 0$, where $(1,-1)$ is an outward normal vector exposing $[ \pi( \overline{\sigma}(v_3)), \pi( \overline{\sigma}(v_7))]$. Therefore, $\pi(v)$ can expose only a subset of $[ \pi( \overline{\sigma}(v_3)), \pi( \overline{\sigma}(v_7))] \subset \pi(\Delta_2)$, and consequently, $v$ can expose only a subset of $\Gamma_2 =  \overline{\sigma}([v_3, v_7])\subset\Delta_2$.
    \end{proof}
\end{subsection}
\begin{subsection}{Computation of tropicalizations}
    Recall that $C^*$ is the curve consisting of polynomials $f$ satisfying $f^\prime = x^{d-1}(x-x_1)(x-x_2)(x-x_3)(x-x_4)$, and $f(x) = \frac{x^{d+4}}{d+4} - \frac{x^{d+3}}{d+3}(x_1+x_2+x_3+x_4) + \frac{x^{d+2}}{d+2}(x_1x_2+x_1x_3+x_1x_4+x_2x_3+x_2x_4+x_3x_4) - \frac{x^{d+1}}{d+1}(x_1x_2x_3+x_2x_3x_4+x_3x_4x_1+x_4x_1x_2) + \frac{x^d}{d}(x_1x_2x_3x_4)$, with the conditions $x_1x_2x_3x_4 = 1$, $f(x_1) - f(x_2) = f(x_3) - f(x_4) = 0$ and $x_i \neq x_j$ for $i \neq j$. Let $\overline{x}:=(x_1,x_2,x_3,x_4)$, so that $f(x) = a_4(\overline{x})x^{d+4} + a_3(\overline{x})x^{d+3} + a_2(\overline{x})x^{d+2} + a_1(\overline{x})x^{d+1} + a_0(\overline{x})x^{d}$. Our goal is to determine the valuation vectors $(\operatorname{val}_{a_0}, \operatorname{val}_{a_1}, \operatorname{val}_{a_2}, \operatorname{val}_{a_3}, \operatorname{val}_{a_4})$ for every branch of $C^*$ at infinity. Since $a_0$ and $a_4$ are nonzero constants under our normalization, we have $\operatorname{val}_{a_0} = \operatorname{val}_{a_4} = 0$.

    In the computations below, we use the equivalent forms of equations \eqref{equation_fx1equalfx2} and \eqref{equation_fx3equalfx4}:

    \begin{equation}\label{equation_fx1equalfx2_alternative} 
        f(x_1) - f(x_2) = f_{1,d+4}(x_1,x_2) - (x_3+x_4)f_{1,d+3}(x_1,x_2) + x_3x_4f_{1,d+2}(x_1,x_2) \text{,}
    \end{equation}
    \begin{equation}\label{equation_fx3equalfx4_alternative} 
        f(x_3) - f(x_4) = f_{1,d+4}(x_3,x_4) - (x_1+x_2)f_{1,d+3}(x_3,x_4) + x_1x_2f_{1,d+2}(x_3,x_4) \text{,}
    \end{equation}

    where $f_{1,k}(x,y) = \frac{-(k-2)x^{k} + (k-2)y^k - kxy^{k-1} + kyx^{k-1}}{k(k-1)(k-2)}$.

    \begin{lemmm}\label{lemma_f1f2_xplusy}
        \begin{enumerate}
            \item For odd $k$, $f_{1,k}(x,-x) = \frac{-4x^k}{k(k-2)}$.
            \item For even $k$, $f_{1,k}(x,y)$ is divisible by $x+y$. Define $f_{2,k-1}(x,y):=\frac{f_{1,k}(x,y)}{x+y}$.
            \item For even $k$, $f_{2,k-1}(x,-x) = \frac{-2x^{k-1}}{k-1}$.
            \item $f_{1,k}(x,y)$ is divisible by $(x-y)^3$, but not by $(x-y)^4$ for $k \geq 3$.
            \item\label{lemma_f1f2_xplusy_item5} $f_{1,k}(x,y)$ has no multiple projective roots except $(1:1)$.
            \item\label{lemma_f1f2_xplusy_item6} $f_{1,k}(x,y)$ and $f_{1,k+1}(x,y)$ have no common projective roots except $(1:1)$.

            \item $\left(\frac{f_{1,k+1}^2(x,y)}{(k+1)k}-\frac{(x+y)f_{1,k+1}(x,y)f_{1,k}(x,y)}{k(k-1)}+\frac{xyf_{1,k}^2(x,y)}{(k-1)(k-2)}\right) = 0$ has a root $x=y$ of multiplicity $6$ and $2k-4$ simple projective roots.
        \end{enumerate}
    \end{lemmm}

    \begin{proof}
        The first two parts follow by direct substitution. Indeed, for even $k$, $f_{1,k}(x,-x) = 0$ , which proves the divisibility by $x + y$.

        For the third part, consider the expression $x^a - y^a$ for $a = 2^lb$, where $b$ is odd and $l > 0$. $x^a - y^a = (x^{b} - y^{b})(x^b+y^b)\prod\limits_{i=1}^{l-1}(x^{2^ib} +y^{2^ib})$. After dividing by $x+y$, we substitute $y = -x$ to obtain $2\cdot b\cdot 2^{l-1}x^{a-1} = ax^{a-1}$ if $l > 0$. 
        
        $f_{1,k}$ may be rewritten as $f_{1,k}(x,y)= \frac{(k-2)(y^k-x^k)+kxy(x^{k-2}-y^{k-2})}{k(k-1)(k-2)}$. Therefore, $f_{2,k-1}(x,-x) =\frac{-(k-2)kx^{k-1}-k(k-2)x^{k-1}}{k(k-1)(k-2)} =\frac{-2x^{k-1}}{k-1}$.

         For the fourth part, consider $F_{1,k}(x) = k(k-1)(k-2)f_{1,k}(x,1)$. For $k \ge 3$,
         \begin{itemize}
             \item $F_{1,k}(1) = -(k-2) + (k-2) - k + k = 0$;
             \item $F_{1,k}^\prime(1) = -k(k-2)+0-k+k(k-1) = 0$;
             \item $F^{\prime\prime}_{1,k}(1) = -k(k-1)(k-2) + 0 - 0 + k(k-1)(k-2) = 0$;
             \item $F_{1,k}^{\prime\prime\prime}(1) = -k(k-1)(k-2)^2 + k(k-1)(k-2)(k-3) = -k(k-1)(k-2)$. 
         \end{itemize}
          Hence, $F_{1,k}$ has a root $1$ of multiplicity exactly $3$. Therefore, $f_{1,k}(x,y)$ is divisible by $(x-y)^3$, but not by $(x-y)^4$.

        For the parts \ref{lemma_f1f2_xplusy_item5} and \ref{lemma_f1f2_xplusy_item6}, we again consider $F_{1,k}(x)$. $F_{1,k}(x)^\prime = -k(k-2)x^{k-1} + k + k(k-1)x^{k-2}$ and $F_{1,k+1}(x) = -(k-1)x^{k+1} + (k-1) -(k+1)x +  (k+1)x^k $. For $k \ge 3$, by the Euclidean algorithm we obtain $\operatorname{gcd}(F_{1,k},F_{1,k}^\prime) = (x-1)^2$ and $\operatorname{gcd}(F_{1,k},F_{1,k+1})=(x-1)^3$.

        For the last part multiply the equation by $(k+1)^3k^3(k-1)^3(k-2)^3$ and substitute $y = 1$. From the lemmas above it is clear that $x = 1$ is a root of multiplicity exactly $6$. Denote this polynomial by $Q(x)$. After substitution $k = d+3$ and $x^d = u$ we obtain the polynomial $H_1(u,x) = u^2A_1(x) + uA_2(x) + A_3(x)$. The same may be done for the derivative of $Q(x)$; we obtain $H_2(u,x) = u^2A_4(x) + uA_5(x) +A_6(x)$. Now we consider the common roots of $H_1$ and $H_2$. Compute resultants $\operatorname{Res}_u(H_1,H_2) = 16(d + 4)^5(d + 3)^5(d + 2)^5(d + 1)^5\cdot((d+4)x - (d+1))^2\cdot((d+1)x - (d + 4))^2\cdot(x - 1)^{10}\cdot x^7$ and $\operatorname{Res}_x(H_1,H_2) = 2^{16}3^9(d + 4)^{20}(d + 3)^{18}(d + 2)^{17}(d + 1)^{24}\cdot(u - 1)^{10}\cdot u^{11}$. Since the first resultant is $0$ only for $x \in \{\frac{d+4}{d+1}, \frac{d+1}{d+4}, 0, 1\}$. The second equation and $u = x^d$ restricts this set to $x \in \{0,1\}$. Since $x = 0$ is not the root of the initial equation, the only possible multiple root is $x = 1$.
    
    \end{proof}
\begin{subsubsection}{Two-dimensional faces}
    We now describe the tropicalizations of the curve $C^*$ at the points at infinity associated with two-dimensional compatible faces of the Newton polytopes of equations defining $C^*$.

    By the symmetries described in Proposition \ref{proposition_compatible_faces_description}, it suffices to consider the case $(\operatorname{val}_{x_1},\operatorname{val}_{x_2}, \operatorname{val}_{x_3}, \operatorname{val}_{x_4}) = (1,1,1,-3)$.

    \begin{prop}\label{prop_2dim}
    The tropicalization of a branch of $C^*$at infinity with valuation vector $(\operatorname{val}_{x_1}, \operatorname{val}_{x_2}, \operatorname{val}_{x_3}, \operatorname{val}_{x_4}) = (1, 1, 1, -3)$ is $0\odot x^d \oplus 3 \odot x^{d+1} \oplus 2 \odot x^{d+2} \oplus 1 \odot x^{d+3} \oplus 0 \odot x^{d+4}$. 
    \end{prop}
    
    \begin{remm}
        For even $d$ this tropicalization occurs only for the class $C_2$ defined in Proposition \ref{proposition_topological_classes}. Indeed, in class $C_1$ the corresponding branches satisfy 
        $\operatorname{val}_{a_3} = \operatorname{val}_{a_1} = -\infty$.
        
    \end{remm}

    \begin{proof}
        By the normalization above, $\operatorname{val}_{a_0} = \operatorname{val}_{a_4} = 0$. Also $(\operatorname{val}_{x_1},\operatorname{val}_{x_2}, \operatorname{val}_{x_3}, \operatorname{val}_{x_4}) = (1,1,1,-3)$ implies that $\operatorname{val}_{a_1} = 3$, since in that case the leading term of $a_1(\overline{x})$ is a nonzero multiple of $x_1x_2x_3$. The leading terms of $a_3(\overline{x})$ and $a_2(\overline{x})$ are proportional to $x_1+x_2+x_3$ and $x_1x_2 + x_2x_3 + x_3x_1$ respectively. We will show that $\operatorname{val}_{(x_1 + x_2 + x_3)} = 1$ and $ \operatorname{val}_{(x_1x_2 + x_2x_3 + x_3x_1)} = 2$. To prove this, it suffices to show that $\operatorname{in}_{x_1} + \operatorname{in}_{x_2}+\operatorname{in}_{x_3} \neq 0, \operatorname{in}_{x_1x_2} + \operatorname{in}_{x_2x_3}+\operatorname{in}_{x_3x_1} \neq 0$.

    Proposition \ref{proposition_compatible_faces_description} and Theorem \ref{theorem_restriction_on_faces} imply that $\operatorname{in}_{x_1}, \operatorname{in}_{x_2}, \operatorname{in}_{x_3}, \operatorname{in}_{x_4}$ give a solution of the system

    \begin{equation}\label{equation_2dim_original_sysytem}
       \begin{cases}
        \frac{f(x_1)-f(x_2)}{(x_1-x_2)^3} |_{\Gamma_1} = 0\text{,} \\
        \frac{f(x_3)-f(x_4)}{(x_3-x_4)^3} |_{\Gamma_2} = 0\text{,}
    \end{cases} 
    \end{equation}

    where $\Gamma_1$ and $\Gamma_2$ come from Proposition \ref{proposition_compatible_faces_description}. 

    Multiplying the first equation by $(x_1-x_2)^3$ and dividing the second equation by $x_3^{d-1}$, we obtain the system
    \begin{equation}\label{equation_2dimensional_zero}
        \begin{cases}
        f_{1, d+4}(x_1,x_2) - x_3f_{1, d+3}(x_1,x_2) = 0\text{,} \\
        \dfrac{x_3^{2}}{(d+4)(d+3)} - \dfrac{(x_1+x_2)x_3}{(d+3)(d+2)} + \dfrac{x_1x_2}{(d+2)(d+1)} = 0 \text{.}
        \end{cases}
    \end{equation}

    The initial coefficients $\operatorname{in}_{x_1}, \operatorname{in}_{x_2}, \operatorname{in}_{x_3}$ satisfy system \eqref{equation_2dimensional_zero}. Therefore, to prove $\operatorname{val}_{a_3} = 1, \operatorname{val}_{a_2} = 2$, it suffices to show that the systems 

    \begin{equation}\label{equation_2dimensional_first}
        \begin{cases}
            f_{1, d+4}(x_1,x_2) - x_3f_{1, d+3}(x_1,x_2) = 0\text{,} \\
        \dfrac{x_3^{2}}{(d+4)(d+3)} - \dfrac{(x_1+x_2)x_3}{(d+3)(d+2)} + \dfrac{x_1x_2}{(d+2)(d+1)} = 0 \text{,} \\
        x_1+x_2+x_3 = 0
        \end{cases}
    \end{equation}

    \begin{equation}\label{equation_2dimensional_second}
        \begin{cases}
            f_{1, d+4}(x_1,x_2) - x_3f_{1, d+3}(x_1,x_2) = 0\text{,} \\
        \dfrac{x_3^{2}}{(d+4)(d+3)} - \dfrac{(x_1+x_2)x_3}{(d+3)(d+2)} + \dfrac{x_1x_2}{(d+2)(d+1)} = 0 \text{,} \\
        x_1x_2+x_2x_3+x_3x_1 = 0
        \end{cases}
    \end{equation}

    have no solutions.

    To simplify the system \eqref{equation_2dimensional_first}, we proceed as follows:

    \begin{itemize}
        \item substitute $x_3 = - x_1 - x_2$ from the third equation to the other equations;
        \item multiply the first equation by $(d+4)(d+3)(d+2)(d+1)$;
        \item multiply the second equation by $(d+4)(d+2)(d+1)$.
    \end{itemize}

    To simplify the system \eqref{equation_2dimensional_second}, we proceed as follows:

    \begin{itemize}
        \item  substitute $x_3 = - \frac{x_1 x_2}{x_1 + x_2}$ from the third equation to the other equations ($x_1+x_2 \neq 0$, otherwise $x_1x_2+x_2x_3+x_3x_1 = -x_1^2\neq 0$);
        \item multiply the first equation by $(d+4)(d+3)(d+2)(d+1)(x_1+x_2)$;
        \item multiply the second equation by $\frac{(d+4)(d+3)(d+1)(x_1+x_2)^2}{x_1x_2}$.
    \end{itemize}

    As a result, we obtain

    \begin{equation}\label{equation_dimension2_final_first}
        \begin{cases}
             (d+3)(d+4)(x_1^{d+3}x_2 + x_1^{d+2}x_2^2 - x_1x_2^{d+3} - x_1^2x_2^{d+2})
            +2(d+1)(d+3)(x_2^{d+4}-x_1^{d+4}) = 0 \text{,}\\
            2(d+1)x_1^2+(5d+8)x_1x_2+2(d+1)x_2^2 = 0 \text{,}
        \end{cases}
    \end{equation}

    \begin{equation}\label{equation_dimension2_final_second}
        \begin{cases}
             2(d+2)(d+4)(x_1^{d+3}x_2^2 - x_1^2x_2^{d+3} )
            +(d+1)(d+2)(x_2^{d+5} + x_1x_2^{d+4} - x_1^{d+4}x_2 -x_1^{d+5} ) = 0\text{,} \\
            2(d+4)x_1^2+(5d+17)x_1x_2+2(d+4)x_2^2 = 0 \text{.}
        \end{cases}
    \end{equation}

    These systems are considered in Lemma \ref{lemma_2dim_system_unsolvable}.

    In order to show that at least one branch of $C^*$ at infinity with valuation vector \\ $(\operatorname{val}_{x_1}, \operatorname{val}_{x_2}, \operatorname{val}_{x_3}, \operatorname{val}_{x_4}) = (1, 1, 1, -3)$ exists, consider the system \eqref{equation_2dimensional_zero} and substitute $x_3 = \frac{f_{1, d+4}(x_1,x_2)}{f_{1, d+3}(x_1,x_2)}$ in the second equation, multiplied by $f_{1, d+3}^2(x_1,x_2)$, to obtain a homogeneous equation

    \begin{equation}\label{equation_2dim_finite_root_substitute}
        \frac{f_{1,d+4}^2(x_1,x_2)}{(d+4)(d+3)}-\frac{(x_1+x_2)f_{1,d+4}(x_1,x_2)f_{1,d+3}(x_1,x_2)}{(d+3)(d+2)}+\frac{x_1x_2f_{1,d+3}^2(x_1,x_2)}{(d+2)(d+1)} = 0\text{.}
    \end{equation}

By Lemma \ref{lemma_f1f2_xplusy}, the left-hand side of \eqref{equation_2dim_finite_root_substitute} has $2d+2$ simple projective roots $(1, r_i)$ which are not equal to $(1,1)$. The same lemma states that $f_{1,d+3}$ and $f_{1,d+4}$ have no common projective roots except $(1:1)$; hence $f_{1,d+3}(1,r_i), f_{1,d+4}(1,r_i)\neq 0$. Therefore, the first equation uniquely determines 

\begin{equation*}
    x_3 = \frac{f_{1, d+4}(1,r_i)}{f_{1, d+3}(1,r_i)} 
\end{equation*}

For each root $r_i$.Since $r_i$ are simple roots, the Jacobian of the initial-form system in the affine chart $x_1 = 1$ has rank $2$. The chart $x_1 = 1$ is transverse to the $C^*$-weight orbit of the quotient torus $\CC^*\PP_{(1,1,1)}$. Thus the hypotheses of the converse part of Theorem \ref{theorem_restriction_on_faces} are satisfied, and there exists a branch at infinity with valuation vector $(1,1,1,-3)$.

    \end{proof}
    
    The following lemma completes the proof of Proposition \ref{prop_2dim}.   
    \begin{lemmm}\label{lemma_2dim_system_unsolvable}
        Systems \eqref{equation_dimension2_final_first} and \eqref{equation_dimension2_final_second} have no solutions.
    \end{lemmm}

    \begin{proof}
    First, rewrite these equations in terms of the variable $x = \frac{x_2}{x_1}$ and eliminate the terms of degrees ${d+4}$ and ${0}$ in the first equation in the first system and of degrees ${d+3}$ and $2$ in the second system to obtain

    \begin{equation*}
    \begin{cases}
          g_1(x) = -6(d+3)(d+2)x^{d+3} - 3(d+3)(d+2)x^{d+2} + 3(d+3)(d+2)x^{2} + 6(d+3)(d+2)x = 0\\
        h_1(x) = 2(d+1)x^2+(5d+8)x+2(d+1) = 0
    \end{cases}
    \end{equation*}

    \begin{equation*}
    \begin{cases}
         g_2(x) = 3(d+3)(d+2)x^{d+5} + 6(d+3)(d+2)x^{d+4} - 6(d+3)(d+2)x - 3(d+3)(d+2) = 0 \\
        h_2(x) = 2(d+4)x^2+(5d+17)x+2(d+4) = 0 
    \end{cases}
    \end{equation*}

    $g_1(x)/x$ and $g_2(x)$ are anti-reciprocal polynomials (a polynomial $p(x)$ is called anti-reciprocal if $p(x) = -x^{\deg p}p(x^{-1})$). Hence, every nonzero root $x$ of $g_1$ or $g_2$ is accompanied by the reciprocal root $x^{-1}$. By Vieta's formulas, the product of the two roots of $h_1(x)$ is equal to $1$; the same holds for $h_2(x)$. Since the two roots of both $h_1$ and $h_2$ are reciprocal, it remains to prove that the root $r_1 = \frac{-(5d+8) + \sqrt{(5d+8)^2-16(d+1)^2}}{4(d+1)}$ of $h_1(x)$ is not a root of $g_1(x)$ and a root $r_2 =\frac{-(5d+17) - \sqrt{(5d+17)^2-16(d+4)^2}}{4(d+4)}$ of $h_2(x)$ is not a root of $g_2(x)$. 
    
    The following estimates hold:
    
    \begin{itemize}
        \item $ 0 > r_1 > \frac{-(5d+8) + (3d+6)}{4(d+1)} = -0.5$;
        \item $r_2 > \frac{-(5d+17) - (3d + 7)}{4(d+4)} = \frac{-(8d+32)+8}{4(d+4)} = -2 + \frac{2}{d+4}$;
        \item $r_2 < \frac{-(5d+17) - (3d + 3)}{4(d+4)} = \frac{-(8d+32)+12}{4(d+4)} = 
        -2 + \frac{3}{d+4} < 0$.
    \end{itemize} 

    $g_1(x) = 3(d+3)(d+2)(-2x^{d+3} -x^{d+2 }+ x^{2} + 2x)$. The following inequalities hold: 
    
    \begin{itemize}
        \item $|2r_1^{d+3}| < |0.5r_1^{d+1}| < 0.5|r_1|$;
        \item $|r_1^{d+2}|<|0.5r_1^{d+1}|<0.5|r_1|$;
        \item $|r_1^2| < 0.5|r_1|$;
        \item $|-2r_1| > 1.5|r_1|$.
    \end{itemize} Therefore, $|-2r_1| > 1.5|r_1|>|-2r_1^{d+3} -r_1^{d+2 }+ r_1^{2}|$. Hence, $r_1$ is not a root of $g_1(x)$.
    \\
    $g_2(x) = 3(d+3)(d+2)(x^{d+5} + 2x^{d+4}-  2x - 1) $. The following inequalities hold: 

    \begin{itemize}
        \item $|-2r_2-1| < 3$;
        \item $|r_2^{d+5} + 2r_2^{d+4}| = |2 + r_2|\cdot|r_2|^{d+4} > \frac{2}{d+4}(2 - \frac{3}{d + 4})^{d+4} =  \frac{2}{d+4}(1 + \frac{d+1}{d + 4})^{d+4} $ \text{.}
    \end{itemize}

    By Bernoulli's inequality,   
    \begin{equation*}
    \begin{split}
        \frac{2}{d+4}\left(1 + \frac{d+1}{d + 4}\right)^{d+4} = \frac{2}{d+4}\left(1 + \frac{d+1}{d + 4}\right)^{\frac{(d+4)(d+2)}{2(d+1)}}\left(\left(1 + \frac{d+1}{d + 4}\right)^{\frac{(d+4)d}{4(d+1)}}\right)^2 \ge \\ 
        \ge \frac{2}{d+4} \cdot \left(1 + \frac{d+2}{2}\right)\left(1 + \frac{d}{4}\right)^2 \ge \left(1 + \frac{d}{4}\right)^2 \ge \frac{49}{16} > 3 \text{.}
    \end{split}
    \end{equation*}

     Therefore, $|r_2^{d+5} + 2r_2^{d+4}| > |-r_2-2|$. Hence, $r_2$ is not a root of $g_2(x)$.
    \end{proof}

\end{subsubsection}

\begin{subsubsection}{One-dimensional skew faces}

We now compute the tropicalizations of the curve $C^*$ at the points at infinity associated with one-dimensional compatible faces of the Newton polytopes of equations defining $C^*$.

    By the symmetries described in Proposition \ref{proposition_compatible_faces_description}, it suffices to consider the case $(\operatorname{val}_{x_1},\operatorname{val}_{x_2}, \operatorname{val}_{x_3}, \operatorname{val}_{x_4}) = (1,1,-1,-1)$.
    
    By the normalization above, $\operatorname{val}_{a_0} = \operatorname{val}_{a_4} = 0$. Also $(\operatorname{val}_{x_1},\operatorname{val}_{x_2}, \operatorname{val}_{x_3}, \operatorname{val}_{x_4}) = (1,1,-1,-1)$ implies that $\operatorname{val}_{a_2} = 2$, since in that case the leading term of $a_2(\overline{x})$ is proportional to $x_1x_2$. The leading terms of $a_3(\overline{x})$ and $a_1(\overline{x})$ are proportional to $x_1+x_2$ and $x_1x_2(x_3+x_4)$ respectively. The expected valuations of these terms are $\operatorname{val}_{(x_1 + x_2)} = 1$ and $ \operatorname{val}_{x_1x_2(x_3+x_4)} = 1$; the last condition is equivalent to $\operatorname{val}_{{(x_3+x_4)} = -1}$. To prove this, it suffices to show that $\operatorname{in}_{x_1} + \operatorname{in}_{x_2}\neq 0, \operatorname{in}_{x_3}+\operatorname{in}_{x_4} \neq 0$. 

    Proposition \ref{proposition_compatible_faces_description} and Theorem \ref{theorem_restriction_on_faces} imply that the initial coefficients $\operatorname{in}_{x_1}, \operatorname{in}_{x_2}, \operatorname{in}_{x_3}, \operatorname{in}_{x_4}$ satisfy the corresponding initial-form equations

    \begin{equation*}
       \begin{cases}
        \frac{f(x_1)-f(x_2)}{(x_1-x_2)^3} |_{\Gamma_1} = 0\text{,} \\
        \frac{f(x_3)-f(x_4)}{(x_3-x_4)^3} |_{\Gamma_2} = 0\text{,}
    \end{cases} 
    \end{equation*}

    where $\Gamma_1$ and $\Gamma_2$ come from Proposition \ref{proposition_compatible_faces_description}.

    Multiply the first equation of the system by $(x_1-x_2)^3$ and the second one by $(x_3-x_4)^3$ to obtain

    \begin{equation}
       \begin{cases}\label{equation_1dimensional_zero}
        f_{1,d+4}(x_1,x_2) = 0\text{,} \\
        x_1x_2f_{1, d + 2}(x_3,x_4) = 0\text{.}
    \end{cases} 
    \end{equation}

    \begin{lemmm}\label{lemma_1dim_finite_roots}
        Each solution of the system \eqref{equation_1dimensional_zero} subject to the condition $x_4 = \frac{1}{x_1x_2x_3}$ is isolated in the quotient torus $\CC^*\PP^2_{(\operatorname{val}_{x_1}, \operatorname{val}_{x_2}, \operatorname{val}_{x_3})}$. All solutions satisfying $x_1\neq x_2$ and $x_3 \neq x_4$ satisfy the hypothesis of the converse part of Theorem \ref{theorem_restriction_on_faces}.
    \end{lemmm}

    \begin{proof}
        Let $A$ be the set of solutions of  $f_{1,d+4}(x_1, x_2) = 0$ in the corresponding weighted projective line, and let $B$ be the analogous set of solutions of $f_{1,d+2}(x_3, x_4) = 0$. Then $|A| \le d+4$ and $|B| \le d+2$. For each pair $(a,b) \in A \times B$, the conditions $a = \frac{x_2}{x_1}$ and $b = x_1x_2x_3^2$ determine at most two points in the quotient torus $(1: a:\pm \sqrt{\frac{b}{a}}) \in \CC^*\PP^2_{(\operatorname{val}_{x_1}, \operatorname{val}_{x_2}, \operatorname{val}_{x_3})}$. Hence, the total number of solutions does not exceed $2(d+4)(d+2)$ and consequently each solution is isolated. By Lemma \ref{lemma_f1f2_xplusy}, each projective root of $f_{1,d+4}$ and $f_{1,d+2}$, except $(1:1)$, is simple. Thus, in the affine chart $x_1 = 1$, the Jacobian of the initial system has rank $2$. This chart is transverse to the $C^*$-weight orbit; hence each solution of \eqref{equation_1dimensional_zero} satisfying $x_1\neq x_2$ and $x_3 \neq x_4$ satisfies the hypothesis of the converse part of Theorem \ref{theorem_restriction_on_faces}.

    \end{proof}

    \begin{prop}\label{proposition_1diminsional_skew_odd}
        For odd $d$ the tropicalization of a branch of $C^*$ at infinity with valuation vector $(\operatorname{val}_{x_1}, \operatorname{val}_{x_2}, \operatorname{val}_{x_3}, \operatorname{val}_{x_4}) = (1, 1, -1, -1)$ is equal to $0 \odot x^d \oplus 1 \odot x^{d+1} \oplus 2 \odot x^{d+2} \oplus 1 \odot x^{d+3}  \oplus 0 \odot x^{d+4}$. 
    \end{prop}

    \begin{proof}
        Applying Lemma \ref{lemma_f1f2_xplusy} in the case of odd $d$, we obtain that
        the system \eqref{equation_1dimensional_zero} has only solutions satisfying $x_1 + x_2 \neq 0$ and  $x_3 + x_4 \neq 0$. Moreover, the same lemma implies that the initial-form system has at least one solution such that $x_1\neq x_2$ and $x_3 \neq x_4$. Therefore, $\operatorname{val}_{x_1+x_2} = 1$ and $\operatorname{val}_{x_3 + x_4} = -1$. Hence, $\operatorname{val}_{a_3} = \operatorname{val}_{(x_1+x_2+x_3+x_4)} =  \operatorname{val}_{(x_1+x_2)} = 1$. Since $\operatorname{val}_{x_1x_2(x_3+x_4)} = 1$ and $\operatorname{val}_{x_3x_4(x_1+x_2)} = -1$, we have $\operatorname{val}_{a_1}= \operatorname{val}_{x_1x_2(x_3+x_4) + x_3x_4(x_1+x_2)} = \operatorname{val}_{x_1x_2(x_3+x_4)} = 1$. By Lemma \ref{lemma_1dim_finite_roots}, this solution of the initial-form system satisfies the hypothesis of the converse part of Theorem \ref{theorem_restriction_on_faces}. Hence, the theorem implies the existence of the required branch.
    \end{proof}

    For even $d$, we use the following equivalent forms of equations \eqref{equation_fx1equalfx2_alternative} and \eqref{equation_fx3equalfx4_alternative}: 

    \begin{equation*}
    \begin{split}
         0 = F_1(x_1,x_2,x_3,x_4) = f(x_1) - f(x_2) =\\ =(x_1+x_2)f_{2, d+3}(x_1,x_2) - (x_3 + x_4)f_{1,d+3}(x_1,x_2)+x_3x_4(x_1+x_2)f_{2, d +1}(x_1,x_2)
    \end{split}
    \end{equation*}

    \begin{equation*}
    \begin{split}
        0 = F_2(x_1,x_2,x_3,x_4) = f(x_3) - f(x_4) = \\ = (x_3+x_4)f_{2, d+3}(x_3,x_4) - (x_1 + x_2)f_{1,d+3}(x_3,x_4)+x_1x_2(x_3+x_4)f_{2, d +1}(x_3,x_4) \text{.}
        \end{split}
    \end{equation*}

    \begin{prop}\label{proposition_1diminsional_skew_even}
        For even $d$, the possible cases are as follows.

        \begin{enumerate}
            \item Assume that $\operatorname{val}_{(x_1 + x_2)} < 1$ and $\operatorname{val}_{(x_3 + x_4)} < -1$. Then $x_1 = -x_2$, $x_3 = - x_4$.
            \item Assume that $\operatorname{val}_{(x_1 + x_2)} = 1$ and $\operatorname{val}_{(x_3 + x_4)} < -1$. Then $\operatorname{val}_{x_3 + x_4} = -3$ and $\operatorname{val}_{a_1} = -1$.
            \item Assume that $\operatorname{val}_{(x_1 + x_2)} < 1$ and $\operatorname{val}_{(x_3 + x_4)} = -1$. Then  $\operatorname{val}_{x_1 + x_2} = -1$ and $\operatorname{val}_{a_3} = -1$.
            \item Assume that $\operatorname{val}_{(x_1 + x_2)} = 1$ and $\operatorname{val}_{(x_3 + x_4)} = -1$. Then  $\operatorname{val}_{a_3} = \operatorname{val}_{a_1} = 1$.
            \item Branches corresponding to all four cases occur. Case 1 corresponds to $C_1$, while the other three cases correspond to $C_2$, where $C_1$ and $C_2$ are the topological equivalence classes defined in Proposition \ref{proposition_topological_classes}.
        \end{enumerate}
    \end{prop}

    \begin{proof}

        1. If $\operatorname{val}_{(x_1 + x_2)} < 1$ and $\operatorname{val}_{(x_3 + x_4)} < -1$, then Lemma \ref{lemma_f1f2_xplusy} implies that 
        
        \begin{itemize}
            \item $\operatorname{val}_{f_{2, d+3}(x_1,x_2)} = d+3$;
            \item $\operatorname{val}_{f_{2, d+1}(x_1,x_2)} = d+1$;
            \item $\operatorname{val}_{f_{1, d+3}(x_1,x_2)} = d+3$;
            \item $\operatorname{val}_{f_{2, d+3}(x_3,x_4)}= -(d+3)$;
            \item $\operatorname{val}_{f_{2, d+1}(x_3,x_4)} = -(d+1)$;
            \item $\operatorname{val}_{f_{1, d+3}(x_3,x_4)} = -(d+3)$.
        \end{itemize}
        Assume first that $\operatorname{val}_{x_1+x_2}$ and $\operatorname{val}_{x_3+x_4}$ are finite (i.e., the corresponding sums are nonzero). Then
        \begin{equation*}
            \operatorname{val}_{ (x_1+x_2)f_{2, d+3}(x_1,x_2)} > \operatorname{val}_{x_3x_4(x_1+x_2)f_{2, d +1}(x_1,x_2)}
        \end{equation*} 
        
        and
        \begin{equation*}
            \operatorname{val}_{x_1x_2(x_3+x_4)f_{2, d +1}(x_3,x_4)} > \operatorname{val}_{(x_3+x_4)f_{2, d+3}(x_3,x_4)} \text{.}
        \end{equation*}
         Since $F_1 = F_2 = 0$, cancellation of the leading terms is necessary, which gives

        \begin{equation*}
            \begin{cases}
                \operatorname{val}_{(x_1+x_2)f_{2, d+3}(x_1,x_2)} = \operatorname{val}_{(x_3 + x_4)f_{1,d+3}(x_1,x_2)} \\
                \operatorname{val}_{x_1x_2(x_3+x_4)f_{2, d +1}(x_3,x_4)} = \operatorname{val}_{(x_1 + x_2)f_{1,d+3}(x_3,x_4)}\text{,}
            \end{cases}
        \end{equation*}

        but the system
        \begin{equation*}
            \begin{cases}
                \operatorname{val}_{(x_1+x_2)} + d + 3 = \operatorname{val}_{(x_3+x_4)} + d + 3 \\
                2 +\operatorname{val}_{(x_3+x_4)} - (d + 1) = \operatorname{val}_{(x_1+x_2)} - (d + 3)
            \end{cases}
        \end{equation*}
          
          has no solution. Hence, the assumption that both valuations are finite leads to a contradiction. If one of these valuations is $-\infty$, then the other valuation is also $-\infty$, as can be seen by considering $F_1$ and $F_2$. Therefore, $\operatorname{val}_{(x_1+x_2)} = \operatorname{val}_{(x_3+x_4)} = -\infty$, or equivalently, $x_1 + x_2 = x_3 + x_4 = 0$. 

        2. The condition $\operatorname{val}_{(x_3+x_4)} < -1$ is equivalent to $\operatorname{in}_{x_3} = -\operatorname{in}_{x_4}$. Since $\operatorname{val}_{(x_1 + x_2)} = 1$, Lemma \ref{lemma_f1f2_xplusy} implies that the following equations are satisfied:

        \begin{itemize}
            \item  $\operatorname{val}_{(x_1 + x_2)f_{1,d+3}(x_3,x_4)} = -d-2$; 
            \item $\operatorname{val}_{f_{2, d+3}(x_3,x_4)} = -d - 3$;
            \item $\operatorname{val}_{x_1x_2f_{2, d +1}(x_3,x_4)} = -d + 1$.
        \end{itemize}

         Since $F_2 = 0$, we have
        \begin{equation*}
            -(d+2) = \operatorname{val}_{(x_1 + x_2)f_{1,d+3}(x_3,x_4)} = \operatorname{val}_{(x_3+x_4)} + \operatorname{val}_{f_{2, d+3}(x_3,x_4)+x_1x_2f_{2, d +1}(x_3,x_4)} = \operatorname{val}_{(x_3+x_4)} - (d - 1) \text{.}
        \end{equation*}
        
        Therefore, $\operatorname{val}_{(x_3+x_4)} = -3$. Since
         \begin{equation*}
            \operatorname{val}_{(x_1 + x_2)f_{1,d+3}(x_3,x_4)}  = \operatorname{val}_{(x_3+x_4)x_1x_2f_{2, d +1}(x_3,x_4)} > \operatorname{val}_{(x_3+x_4)f_{2, d+3}(x_3,x_4)}
         \end{equation*}
          and $F_2=0$, the initial forms of the two leading terms must cancel; hence the equality $\operatorname{in}_{(x_3+x_4)x_1x_2f_{2, d +1}(x_3,x_4)} = \operatorname{in}_{(x_1 + x_2)f_{1,d+3}(x_3,x_4)}$ holds. By Lemma \ref{lemma_f1f2_xplusy}:
        
        \begin{itemize}
            \item $\operatorname{in}_{f_{2, d +1}(x_3,x_4)} = f_{2,d +1}(\operatorname{in}_{x_3}, -\operatorname{in}_{x_3}) = \frac{-2\operatorname{in}_{x_3}^{d+1}}{d+1}$; \\
            \item $\operatorname{in}_{f_{1, d + 3}(x_3,x_4)} = f_{1, d +3}(\operatorname{in}_{x_3}, -\operatorname{in}_{x_3}) = \frac{-4\operatorname{in}_{x_3}^{d+3}}{(d+3)(d+1)}=\operatorname{in}_{x_3x_4}\cdot\frac{4\operatorname{in}_{x_3}^{d+1}}{(d+3)(d+1)}$ \text{.}
        \end{itemize}

        Hence,

        \begin{equation*}
            \operatorname{in}_{(x_3+x_4)x_1x_2} = \frac{-2}{(d+3)}\operatorname{in}_{(x_1+x_2)x_3x_4} \text{,}
        \end{equation*}

        and $\operatorname{in}_{(x_3+x_4)x_1x_2} \neq -\operatorname{in}_{(x_1+x_2)x_3x_4}$ for $d \ge 3$.
        
        3. The condition $\operatorname{val}_{(x_1 + x_2)} < 1$ is equivalent to $\operatorname{in}_{x_1} = -\operatorname{in}_{x_2}$. Since $\operatorname{val}_{x_3 + x_4} = -1$, Lemma \ref{lemma_f1f2_xplusy} implies that the following equations are satisfied:

        \begin{itemize}
            \item  $\operatorname{val}_{(x_3+x_4)f_{1, d+3}(x_1,x_2)} = d+2$;
            \item $\operatorname{val}_{f_{2, d+3}(x_1,x_2)} = d+3$;
            \item $\operatorname{val}_{x_3x_4f_{2, d+1}(x_1,x_2)} = d-1$.
        \end{itemize}

        Since $F_1 = 0$, we have
           \begin{equation*}
            d+2 = \operatorname{val}_{(x_3 + x_4)f_{1,d+3}(x_1,x_2)} = \operatorname{val}_{x_1 + x_2} + \operatorname{val}_{f_{2, d+3}(x_1,x_2)+x_3x_4f_{2, d+1}(x_1,x_2)} = \operatorname{val}_{(x_1+x_2)} + d + 3\text{.}
            \end{equation*}
            Therefore, $\operatorname{val}_{(x_1 + x_2)} = -1$. Since
            \begin{equation*}
                \operatorname{val}_{(x_3+x_4)f_{1, d+3}(x_1,x_2)} = \operatorname{val}_{(x_1 + x_2)f_{2, d+3}(x_1,x_2)} > \operatorname{val}_{x_3x_4f_{2, d+1}(x_1,x_2)}
            \end{equation*}

            and $F_1=0$, the initial forms of the two leading terms must cancel; hence the equality $\operatorname{in}_{(x_3+x_4)f_{1, d+3}(x_1,x_2)} = \operatorname{in}_{(x_1 + x_2)f_{2, d+3}(x_1,x_2)}$ holds. By Lemma \ref{lemma_f1f2_xplusy}:

            \begin{itemize}
                \item $\operatorname{in}_{f_{1, d+3}(x_1,x_2)} = f_{1, d+3}(\operatorname{in}_{x_1}, -\operatorname{in}_{x_1}) = \frac{-4\operatorname{in}_{x_1}^{d+3}}{(d+1)(d+3)}$;
                \item  $\operatorname{in}_{f_{2, d+3}(x_1,x_2)} = f_{2, d+3}(\operatorname{in}_{x_1}, -\operatorname{in}_{x_1}) = \frac{-2\operatorname{in}_{x_1}^{d+3}}{(d+3)}$.
            \end{itemize}

        Hence,

            \begin{equation*}
                \operatorname{in}_{(x_1+x_2)} = \frac{2}{d+1}\operatorname{in}_{(x_3+x_4)}
            \end{equation*}
            
        and $\operatorname{in}_{(x_1+x_2)} \neq -\operatorname{in}_{(x_3+x_4)}$ for $d \ge 3$.

        4. Since $\operatorname{val}_{(x_1+x_2)} = 1$ and $\operatorname{val}_{(x_3 + x_4)} = -1$, we have $\operatorname{val}_{a_3} = \operatorname{val}_{(x_1+x_2+x_3+x_4)} =  \operatorname{val}_{(x_1+x_2)} = 1$. Equalities $\operatorname{val}_{x_1x_2(x_3+x_4)} = 1$ and $\operatorname{val}_{x_3x_4(x_1+x_2)} = -1$ are satisfied. Therefore, $\operatorname{val}_{a_1} = \operatorname{val}_{x_1x_2(x_3+x_4) + x_3x_4(x_1+x_2)} = \operatorname{val}_{x_1x_2(x_3+x_4)} = 1$.

        5. Lemma \ref{lemma_f1f2_xplusy} implies that, for even $d \ge 4$:
        \begin{itemize}
            \item $f_{1,d+4}(x_1,x_2) = (x_1 + x_2)f_{2, d+3}(x_1,x_2) = (x_1+x_2)(x_1-x_2)^3g_d(x_1,x_2)$, for some homogeneous polynomial $g_d$ of degree $d \geq 4$ satisfying $g_d(1,\pm1) \ne 0$.
            \item $f_{1,d+2}(x_3,x_4) = (x_3 + x_4)f_{2, d+1}(x_3,x_4) = (x_3 + x_4)(x_3-x_4)^3g_{d-2}(x_3,x_4)$, for some homogeneous polynomial $g_{d-2}$ of degree $d - 2 \geq 2$ satisfying $g_{d-2}(1,\pm1) \ne 0$.
        \end{itemize} 

        Therefore, $f_{1,d+4}(x_1,x_2)$ and $ f_{1,d+2}(x_3,x_4)$ have projective roots $(1:r_1),(1:r_2) \in \CC^*\PP^1_{(1,1)} \setminus \{(1:-1),(1:1)\}$ such that 
        \begin{itemize}
            \item $f_{1,d+4}(x,r_1x) = 0$,
            \item $f_{1, d+2}(x,r_2x) = 0$,
        \end{itemize}
        for all $x \in \CC^*$. Since $r_1$ and $r_2$ give solutions to the equations in system \eqref{equation_1dimensional_zero} and neither equals $1$, they  give solutions to the equations in the initial-form system.
          
          Then:

        \begin{itemize}
            \item Case 1 corresponds to projective roots $(1:-1)$, $(1:-1)$, and the corresponding branch occurs only for $C_1$ since $\operatorname{val}_{a_3} = \operatorname{val}_{a_1} = -\infty$. Moreover, for each branch corresponding to $C_1$, the property $\operatorname{val}_{a_3} = \operatorname{val}_{a_1} = -\infty$ holds.
            \item Case 2 corresponds to projective roots $(1:r_1)$ and $(1:-1)$.
            \item Case 3 corresponds to projective roots $(1:-1)$ and $(1:r_2)$.
            \item Case 4 corresponds to projective roots $(1:r_1$) and $(1:r_2)$.
        \end{itemize}

        Cases 2--4 correspond to $C_2$ since $\operatorname{val}_{a_3} \neq -\infty$ and $\operatorname{val}_{a_1}\neq -\infty$. Since for all four of these cases we showed the existence of projective roots of the initial-form system, Lemma \ref{lemma_1dim_finite_roots} implies that each of them satisfies the hypothesis of the converse part of Theorem \ref{theorem_restriction_on_faces}. Theorem \ref{theorem_restriction_on_faces} implies that each of these branches occurs.
    \end{proof}

    Combining Proposition \ref{proposition_1diminsional_skew_odd}  with all four cases of Proposition \ref{proposition_1diminsional_skew_even} and using Proposition \ref{proposition_compatible_faces_description} to exclude other compatible faces, we obtain the following complete list of tropicalizations.
    \begin{prop}\label{prop_1dim_skew}
     The distinct tropicalizations of branches at infinity of $C^*$ such that \\ $(\operatorname{val}_{x_1}, \operatorname{val}_{x_2}, \operatorname{val}_{x_3}, \operatorname{val}_{x_4}) = (1, 1, -1, -1)$ are precisely the following:

    \begin{itemize}
        \item $0\odot x^d \oplus 1 \odot x^{d+1} \oplus 2 \odot x^{d+2} \oplus 1 \odot x^{d+3} \oplus 0 \odot x^{d+4}$ for odd $d$;
        \item $0 \odot x^d \oplus (-\infty) \odot x^{d+1} \oplus 2 \odot x^{d+2} \oplus (-\infty) \odot x^{d+3} \oplus 0 \odot x^{d+4}$ for even $d$ and polynomials in the class $C_1$ defined in Proposition \ref{proposition_topological_classes};
        \item $0\odot x^d \oplus 1 \odot x^{d+1} \oplus 2 \odot x^{d+2} \oplus 1 \odot x^{d+3} \oplus 0 \odot x^{d+4}$, $0\odot x^d \oplus 1 \odot x^{d+1} \oplus 2 \odot x^{d+2} \oplus (-1) \odot x^{d+3} \oplus 0 \odot x^{d+4}$ or $0\odot x^d \oplus (-1) \odot x^{d+1} \oplus 2 \odot x^{d+2} \oplus 1 \odot x^{d+3} \oplus 0 \odot x^{d+4}$ for even $d$ and polynomials in the class $C_2$ defined in Proposition \ref{proposition_topological_classes}.
    \end{itemize} 
    
    \end{prop}
\end{subsubsection}

\begin{subsubsection}{One-dimensional collinear faces}
    Here we describe the tropicalization of the curve $C^*$ at the points at infinity associated with compatible collinear faces of the Newton polytopes of equations defining $C^*$.

\begin{prop}\label{prop_1dim_coll}
The smooth compactification of $C^*$ has no points at infinity associated with one-dimensional collinear compatible pairs of faces.
    
\end{prop}

\begin{proof}
    In this case, Proposition \ref{proposition_compatible_faces_description} does not determine a specific valuation vector $(\operatorname{val}_{x_1}, \operatorname{val}_{x_2}, \operatorname{val}_{x_3}, \operatorname{val}_{x_4})$, but Theorem \ref{theorem_restriction_on_faces} implies that $\operatorname{in}_{x_1}, \operatorname{in}_{x_2}, \operatorname{in}_{x_3}, \operatorname{in}_{x_4}$ provide a solution to the system

    \begin{equation*}
        \begin{cases}
            \frac{f(x_1)-f(x_2)}{(x_1-x_2)^3} |_{\Gamma_1} = 0\text{,} \\
        \frac{f(x_3)-f(x_4)}{(x_3-x_4)^3} |_{\Gamma_2} = 0\text{,}
        \end{cases}
    \end{equation*}

    where $\Gamma_1$ and $\Gamma_2$ are described in Proposition \ref{proposition_compatible_faces_description}. This system is equivalent to 

\begin{equation*}
    \begin{cases}
    \frac{x_1^{d+1}}{(d+4)(d+3)} = \frac{x_1^dx_3}{(d+3)(d+2)} \text{,} \\
    \frac{x_3^{d+1}}{(d+4)(d+3)} = \frac{x_3^dx_1}{(d+3)(d+2)}\text{.} 
    \end{cases}
\end{equation*}

Since $x_1,x_3\neq 0$, this system is equivalent to

\begin{equation*}
    \begin{cases}
    (d+2)x_1=(d+4)x_3 \text{,} \\
    (d+2)x_3 = (d+4)x_1\text{.} 
    \end{cases}
\end{equation*}
Hence, $(d+2)^2=(d+4)^2$, which is impossible for $d \ge 3$. The remaining three collinear pairs are obtained from the one considered above by the symmetries of Proposition \ref{proposition_compatible_faces_description}, and hence are excluded in the same way.
\end{proof}

\end{subsubsection}

\begin{subsubsection}{Intersections with coordinate planes of $\CC^n$}

    In previous subsections we described the tropicalization of the set of polynomials with the support $\{d+4,d+3,d+2,d+1,d\}$ and the passport $[[2,2,1,\ldots,1],[2,2,1,\ldots,1]]$ using the parametrization by nonzero critical values. Our goal was to express the valuation vectors of the branches at infinity in terms of the coefficients of the polynomials. But one type of point at infinity is not captured by this approach: intersections with coordinate planes of $\CC^n$. It happens because a polynomial $a_4x^{d+4}+\ldots+a_0x^{d}$ such that one of $a_1, a_2, a_3$ is $0$ corresponds to a polynomial with four nonzero critical values which is a regular point in terms of critical values. In this section we compute the intersection with coordinate planes. 
    
    \begin{prop}\label{prop_intersection_a1=a2=a3=0}
        Let $f$ be a polynomial such that $f^\prime = x^{d-1}(x-x_1)(x-x_2)(x-x_3)(x-x_4) = x^{d-1}(x^4+1)$ and $f(x_1) = f(x_2)$, $f(x_3)= f(x_4)$. Then
        \begin{itemize}
            \item if $d$ is odd there are no such polynomials;
            \item if $d = 4k+2$, then there are $8$ possible ordered quadruples $(x_1,x_2,x_3,x_4)$;
            \item if $d = 4k$,  then there are $24$ possible ordered quadruples $(x_1,x_2,x_3,x_4)$.
        \end{itemize}
    \end{prop}
    \begin{proof}
        $x_i$ are equal to $a(\pm1\pm i)$ for some positive real $a$. Since $x_i^4 = -a^4$, $f(x_i) = bx^m$, where $b$ is a nonzero value and $m \equiv d \operatorname{(mod 4)}$. For odd $d$ the values $f(x_i)$ are four distinct values. For $d = 4k$ all $f(x_i)$ coincide, then any bijection between $x_i$ and $\pm1\pm i$ provides a solution. For $d = 4k + 2$, we have $f(x_i) = bx_i^2$ and $f(x_i) = f(x_j)$ only if $x_i = -x_j$, which provides $8$ solutions. In the even case, $8$ solutions are such that $x_1 + x_2 = 0$ and $x_3+x_4 = 0$ and these $8$ roots correspond to the component $C_1$ and in case $d = 4k$ other $16$ roots correspond to $C_2$. 
    \end{proof}

    \begin{prop}\label{prop_intersection_a1=a3=0}
         Let $f$ be a polynomial such that $f^\prime = x^{d-1}(x-x_1)(x-x_2)(x-x_3)(x-x_4) = x^{d-1}(x^4+ax^2+1)$ and $\frac{f(x_1)-f(x_2)}{(x_1-x_2)^3} =\frac{f(x_3)-f(x_4)}{(x_3-x_4)^3} = 0$. Then
          \begin{itemize}
            \item if $d$ is odd there are $8(d+1)$ solutions $(x_1,x_2,x_3,x_4)$ counted with multiplicities;
            \item if $d$ is even, then there are $8(d-2)$ solutions $(x_1,x_2,x_3,x_4)$ counted with multiplicities such that $x_1 + x_2 \neq 0$;
            \item if $d$ is even, then there are $8$ solutions $(x_1,x_2,x_3,x_4)$ each with multiplicity $4$ such that $x_1+x_2 = 0$ and $\frac{f(x_1)-f(x_2)}{(x_1-x_2)^3(x_1+x_2)} = 0$ or $\frac{f(x_3)-f(x_4)}{(x_3-x_4)^3(x_3+x_4)} = 0$. 
            
        \end{itemize}
    \end{prop}

    \begin{proof}
        Note that, if $x_1+x_2 = -(x_3+x_4)$ then $x_1x_2x_3+x_2x_3x_4 + x_3x_4x_1+x_4x_1x_2=(x_1+x_2)x_3x_4+(x_3+x_4)x_1x_2 = (x_1+x_2)(x_3x_4-x_1x_2)$.
        Consider the case $x_1+x_2 \neq 0$. Then $(x_1+x_2)(x_3x_4-x_1x_2) = 0$ implies that $x_1x_2=x_3x_4$ and consequently
            \begin{equation}\label{equation_reduction_x1x2_equal_x3x4}
            f(x_1)-f(x_2)= f_{1,d+4}(x_1,x_2)+(x_1+x_2)f_{1,d+3}(x_1,x_2)+x_1x_2f_{1,d+2}(x_1,x_2) \text{.}
            \end{equation}
            Let $(1,s_1)$ be a solution of \eqref{equation_reduction_x1x2_equal_x3x4}. Denote $s_2 := \frac{x_4}{x_3}$. Since $x_1x_2x_3x_4 = 1$ and $x_1x_2=x_3x_4$, $x_1^2s_1 = x_3^2s_2 = \pm1$. Moreover, $$x_1+x_2+x_3+x_4 = x_1(1 + s_1) + x_3(1 + s_2)=x_1 \pm \frac{1}{x_1}+x_3 \pm\frac{1}{x_3} = 0\text{.}$$

            Multiplying the last equation by $x_3$, we obtain a quadratic equation with the roots $x_3 = -x_1$ and $x_3 = \mp\frac{1}{x_1}$. In the first case $s_2=s_1$ and in the second case $s_2= \frac{1}{s_1}$. Since \eqref{equation_reduction_x1x2_equal_x3x4} is reciprocal, in both cases $(x_3,x_4)$ is a solution for $f(x_3) = f(x_4)$. Hence, each root $s_1 \neq -1$ of \eqref{equation_reduction_x1x2_equal_x3x4} defines $8$ solutions $$x_1 = \pm \sqrt{\pm \frac{1}{s_1}}, x_2 = x_1s_1, s_2 = s_1^{\pm1}, x_3 = -\frac{x_1(1+s_1)}{(1+s_2)}, x_4 = s_2x_3$$ of the system. Substitute $x_1 = x, x_2 = 1$ in \eqref{equation_reduction_x1x2_equal_x3x4} and multiply by $\frac{d(d+4)}{2}$ to obtain
            \begin{equation*}
                f_{3,d+4}(x) = -dx^{d+4} + (d+4)x^{d+2} - (d+4)x^2 + d = 0 \text{.}
            \end{equation*}

            For odd $d$, $f_{3,d+4}(-1) = -8$.
            For even $d$, we compute the derivatives

            \begin{itemize}
                \item $f_{3, d+4}(-1) = -d + (d+4) - (d+4) + d = 0$;
                \item $f_{3, d+4}^\prime(-1) = d(d+4) - (d+4)(d+2) + 2(d+4) = 0$;
                \item $f_{3, d+4}^{\prime\prime}(-1) = -d(d+3)(d+4) + (d+4)(d+2)(d+1) - 2(d+4) = (d+4)(-d(d+3) + (d+2)(d+1) -2) = 0$;
                \item $f_{3, d+4}^{\prime\prime\prime}(-1) = (d+4)(d+3)(d+2)d - (d+4)(d+2)(d+1)d = 2(d+4)(d+2)d \neq 0$.
            \end{itemize}

            Therefore, for odd $d$ there are $8(d+1)$ solutions with multiplicities and for even $d$ there are $8(d-2)$ solutions.

            In the case $x_1 + x_2 = 0$ for odd $d$, we substitute $x_2 = -x_1$ in $f(x_1) = f(x_2)$ to obtain $-\frac{4x_1^{d+4}}{(d+4)(d+2)} - x_3x_4\frac{4x_1^{d+2}}{(d+2)d} = 0$ by Lemma \ref{lemma_f1f2_xplusy} for odd $d$. Hence, $dx_1x_2 = (d+4)x_3x_4$. Conversely, in the same way we may show that $(d+4)x_1x_2 = dx_3x_4$. Thus, for odd $d$ there are no solutions.
            
            For even $d$, we substitute $x_2 = -x_1$ in $\frac{f(x_1) - f(x_2)}{x_1+x_2} = 0$ to obtain $-\frac{2x_1^{d+3}}{d+3} - \frac{4x_1^{d+3}}{(d+3)(d+1)} - x_3x_4 \frac{2x_1^{d+1}}{d+1}$. Therefore, $x_1x_2 = x_3x_4$ and $f(x_1) = f(x_2)$ becomes equivalent to \eqref{equation_reduction_x1x2_equal_x3x4}. We have
            \begin{itemize}
                \item $x_2 = -x_1$;
                \item $x_4 = -x_3$;
                \item $x_1x_2x_3x_4 = 1$;
                \item $x_1x_2 = x_3x_4$.
            \end{itemize}

            Hence, there are $8$ following solutions $$x_1 = \pm \sqrt{\pm 1}, x_3 = \pm x_1, x_2 = -x_1, x_4 = -x_3 \text{.}$$. If we consider the equations $frac{f(x_1)-f(x_2)}{(x_1-x_2)^3(x_1+x_2)} = 0$ and $\frac{f(x_3)-f(x_4)}{(x_3-x_4)^3(x_3+x_4)} = 0$, the multiplicity of these solutions is $2$, since $(-1)$ is a root of $f_{3,d+4}$ with multiplicity $3$. If we multiply one of them by $(x_1+x_2)$ or $(x_3 + x_4)$ then multiplicity will be increased by $1$. Hence, total multiplicity is $4$. 
    \end{proof}

    \begin{conjecture}\label{conj_intersection_a1=a2=0_or_a3=a2=0=>a1=a2=a3=0}
    \begin{enumerate}
        \item Let $f$ be a polynomial such that $f^\prime = x^{d-1}(x-x_1)(x-x_2)(x-x_3)(x-x_4) = x^{d-1}(x^4+ax+1)$ and $\frac{f(x_1)-f(x_2)}{(x_1-x_2)^3} =\frac{f(x_3)-f(x_4)}{(x_3-x_4)^3} = 0$. Then $a = 0$.
        \item Let $f$ be a polynomial such that $f^\prime = x^{d-1}(x-x_1)(x-x_2)(x-x_3)(x-x_4) = x^{d-1}(x^4+ax^3+1)$ and $f(x_1)-f(x_2) = f(x_3)-f(x_4) = 0$. Then $a = 0$.
    \end{enumerate}
         
    \end{conjecture}

    \begin{remm}
        This conjecture was not proved, but was checked for $d \leq 300$ using SageMath.
    \end{remm}
      
\end{subsubsection}

\begin{subsubsection}{Weighted 1-fan}
    
        We can now describe the weighted 1-fan corresponding to the tropicalizations of $C^*$. Combining Propositions \ref{proposition_topological_classes}, \ref{prop_2dim}, \ref{prop_1dim_skew}, \ref{prop_1dim_coll}, \ref{prop_intersection_a1=a2=a3=0} and \ref{prop_intersection_a1=a3=0} under assumption of Conjecture \ref{conj_intersection_a1=a2=0_or_a3=a2=0=>a1=a2=a3=0} we obtain the following proposition.

        \begin{prop}\label{prop_main}
            Assuming Conjecture \ref{conj_intersection_a1=a2=0_or_a3=a2=0=>a1=a2=a3=0}, the 1-fan of $C^*$ in terms of the coordinates $(a_1,a_2,a_3)$ of polynomials $f(x)= a_4x^{d+4} + a_3x^{d+3} + a_2 x^{d+2} + a_1 x^{d+1} + a_0x^{d}$ is as follows.

            \begin{itemize}
                \item For odd $d$, the fan consists of the following vectors:

                \begin{enumerate}
                    \item $(3,2,1)$ with the weight $4(2d+2)$;
                    \item $(1,2,1)$ with the weight $4(d+1)(d-1)$;
                    \item $(-1,0,-1)$ with the weight $8(d+1)$;
                    \item $(-1,0,0)$ with the weight $4(d+1)(d+3)$;
                    \item $(0, -1, 0)$ with the weight $8(d+1)^2$;
                    \item $(0,0,-1)$ with the weight $4(d+1)(d-1)$.
                \end{enumerate}

                \item For even $d$, the fan corresponding to the class $C_2$ consists of the following vectors:
                     \begin{enumerate}
                    \item $(3, 2, 1)$ with the weight $4(2d+2)$;
                    \item $(1, 2, 1)$ with the weight $4d(d-2)$;
                    \item $(-1, 2, 1)$ with the weight $4d$;
                    \item $(1, 2, -1)$ with the weight $4(d-2)$;

                    \item $(-1, -1, -1)$ with the weight $16$ if $d = 4k$;
                    \item $(-1, 0, -1)$ with the weight $8(d+2)$ if $d = 4k+2$ and $8d$ if $d = 4k$;
                    
                    \item $(-1,0,0)$ with the weight $4d(d+2)$;
                    \item $(0, -1, 0)$ with the weight $8d(d+2)$ if $d = 4k + 2$ and $8(d^2+2d-2)$ if $d = 4k$;
                    \item $(0,0,-1)$ with the weight $4d(d-2)$;

                \end{enumerate}

                \item For even $d$, the fan corresponding to the class $C_1$ consists of the following vectors:
                    \begin{enumerate}
                    \item $(-\infty, 2, -\infty)$ with the weight $4$;
                    \item $(-\infty, -1, -\infty)$ with the weight $8$.
                \end{enumerate}
            \end{itemize}
        \end{prop}

        \begin{proof}
            All the vectors and weights except basis vectors $-e_1, -e_2$ and $-e_3$ were considered in propositions above. These vectors correspond to the cases when only one of the coefficients of the polynomial is $0$ and this is the only type of intersections with the coordinate planes which was not discussed above. The weight of such vectors comes from the condition that the fan should be balanced (i.e. its weighted sum is $0$). Alternatively, these cases can be verified using the Kushnirenko-Bernstein theorem.
        \end{proof}
\end{subsubsection}
\end{subsection}
\end{section}

\begin{section}{Main result}

    Combining Propositions \ref{proposition_topological_classes} and \ref{prop_main} under assumption of Conjecture \ref{conj_intersection_a1=a2=0_or_a3=a2=0=>a1=a2=a3=0} we obtain the following theorem.

    \begin{theorr}{(conditional on Conjecture \ref{conj_intersection_a1=a2=0_or_a3=a2=0=>a1=a2=a3=0})}\label{theorem_main_theorem}
    \begin{enumerate}
        \item Let $d \ge 3$ and let $A = \{d, d+1,d+2,d+3,d+4\}$. For $f \in \CC^A$ with two nonzero critical values, each having two critical points of local degree $2$, the topological equivalence classes are as follows:

        \begin{itemize}
            \item One class for odd $d$. Denote this class by $C_0$.
            \item Two classes for even $d$. One of the two classes consists of polynomials $f$ such that $f^\prime(x) = cx^{d-1}(x^2-x_1^2)(x^2-x_2^2)$. Denote this class by $C_1$ and the other class by $C_2$.
        \end{itemize}

        \item The following tropicalizations occur for the unique class $C_0$ when $d$ is odd and for the class $C_2$ when $d$ is even:
            \begin{itemize}
                \item $0 \odot x^d \oplus 3 \odot x^{d+1} \oplus 2 \odot x^{d+2} \oplus 1 \odot x^{d+3}  \oplus 0 \odot x^{d+4}$;
                \item $0 \odot x^d \oplus 1 \odot x^{d+1} \oplus 2 \odot x^{d+2} \oplus 1 \odot x^{d+3}  \oplus 0 \odot x^{d+4}$;
                \item $0 \odot x^d \oplus (-1) \odot x^{d+1} \oplus 0 \odot x^{d+2} \oplus (-1) \odot x^{d+3}  \oplus 0 \odot x^{d+4}$;
                \item $0 \odot x^d \oplus (-1) \odot x^{d+1} \oplus 0 \odot x^{d+2} \oplus0 \odot x^{d+3}  \oplus 0 \odot x^{d+4}$;
                \item $0 \odot x^d \oplus 0 \odot x^{d+1} \oplus (-1) \odot x^{d+2} \oplus0 \odot x^{d+3}  \oplus 0 \odot x^{d+4}$;
                \item $0 \odot x^d \oplus 0 \odot x^{d+1} \oplus 0 \odot x^{d+2} \oplus (-1) \odot x^{d+3}  \oplus 0 \odot x^{d+4}$.
            \end{itemize}

        \item The list of tropicalizations which occur for branches at infinity only for $C_2$ is as follows:
             \begin{itemize}
                 \item  $0 \odot x^d \oplus (-1) \odot x^{d+1} \oplus (-1) \odot x^{d+2} \oplus (-1) \odot x^{d+3}  \oplus 0 \odot x^{d+4}$ if $d = 4k$;
                 \item $0 \odot x^d \oplus (-1) \odot x^{d+1} \oplus 2 \odot x^{d+2} \oplus 1 \odot x^{d+3}  \oplus 0 \odot x^{d+4}$;
                 \item $0 \odot x^d \oplus 1 \odot x^{d+1} \oplus 2 \odot x^{d+2} \oplus (-1) \odot x^{d+3}  \oplus 0 \odot x^{d+4}$.
             \end{itemize}

        \item The list of tropicalizations which occur for branches at infinity of $C_1$ is as follows:
            \begin{itemize}
                \item $0\odot x^d \oplus (-\infty) \odot x^{d+1} \oplus 2 \odot x^{d+2} \oplus(-\infty) \odot x^{d+3} \oplus 0 \odot x^{d+4}$;
                \item $0\odot x^d \oplus (-\infty) \odot x^{d+1} \oplus (-1) \odot x^{d+2} \oplus(-\infty) \odot x^{d+3} \oplus 0 \odot x^{d+4}$.
            \end{itemize}

    \end{enumerate}

    \end{theorr}

    \begin{remm}\label{remm_to_main_theor}
        For the tropicalizations arising from intersections with coordinate hyperplanes, the valuation vectors are determined only up to multiplication by a positive scalar.
    \end{remm}

    By Remark \ref{remm_to_main_theor}, we do not know the tropicalization precisely, but we have enough information to draw the plots of tropical polynomials. Consider two linear functions $l_1(x) = b+ax$ and $l_2(x) = kb+ax$. Since $k\cdot l_1(x) = l_2(kx)$, the plot of $l_1(x)$ may be obtained from the plot of $l_2(x)$ by a rescaling. Hence, two tropical polynomials whose coefficients differ by a common positive scalar have the same incidence relations among the lines in their plots. Below we show the plots of tropicalizations which occur for classes $C_0$ and $C_2$. 
\pagebreak

\begin{figure}[h]
    \centering
    \begin{subfigure}[b]{0.31\textwidth}
        \centering
        \includegraphics[width=\textwidth]{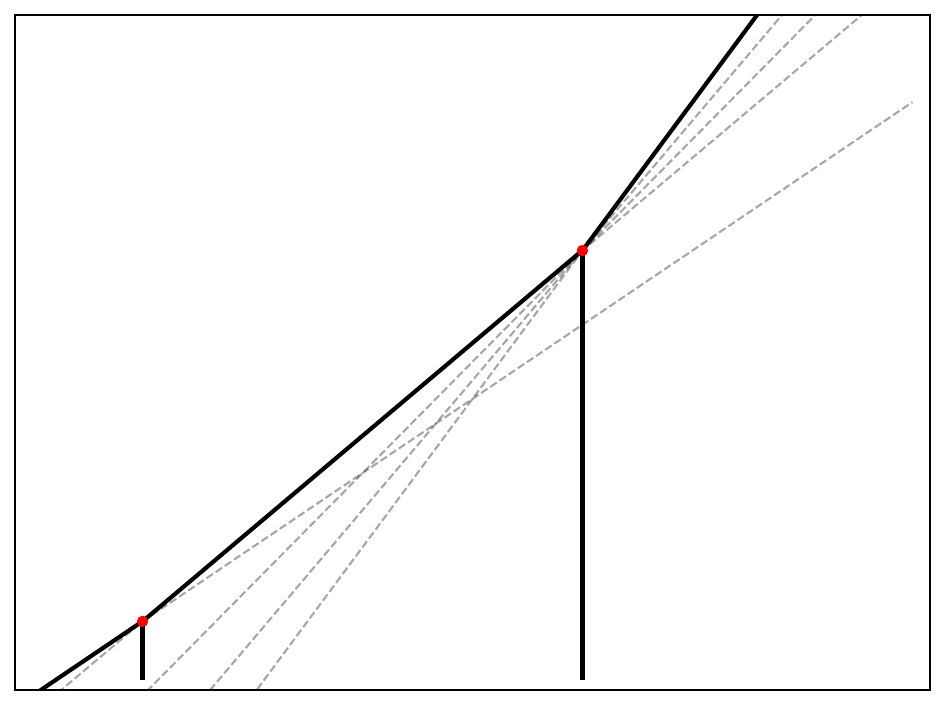}
        \caption{\microfont $0{\odot}x^{d+4} \oplus 1{\odot}x^{d+3} \oplus 2{\odot}x^{d+2} \oplus 3{\odot}x^{d+1} \oplus 0{\odot}x^d$}\label{plot_a}
    \end{subfigure}\hfill
    \begin{subfigure}[b]{0.31\textwidth}
        \centering
        \includegraphics[width=\textwidth]{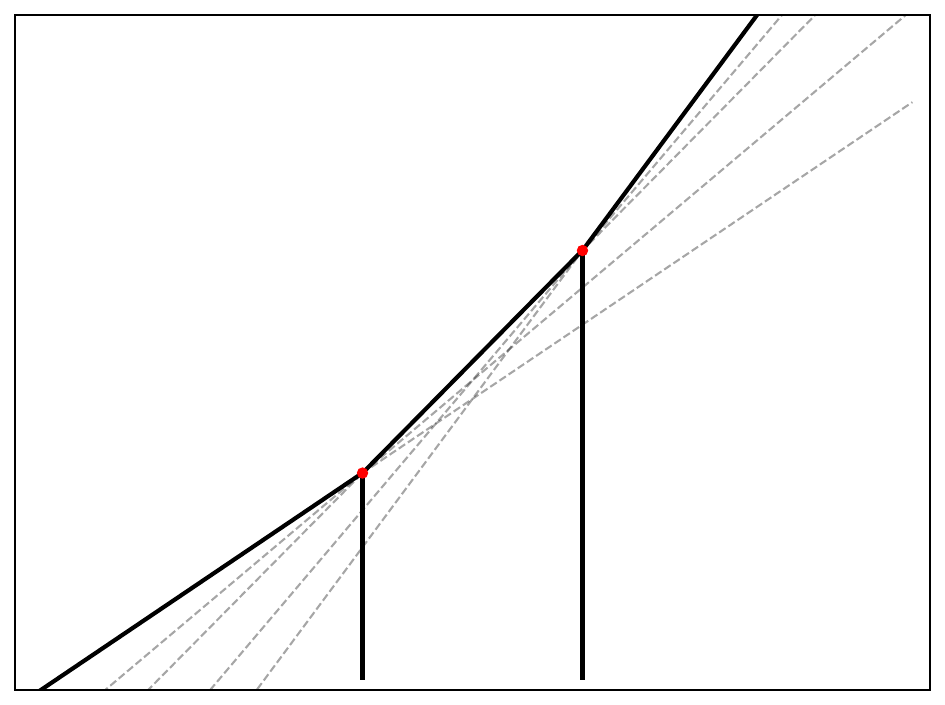}
        \caption{\microfont $0{\odot}x^{d+4} \oplus 1{\odot}x^{d+3} \oplus 2{\odot}x^{d+2} \oplus 1{\odot}x^{d+1} \oplus 0{\odot}x^d$}\label{plot_b}
    \end{subfigure}\hfill
    \begin{subfigure}[b]{0.31\textwidth}
        \centering
        \includegraphics[width=\textwidth]{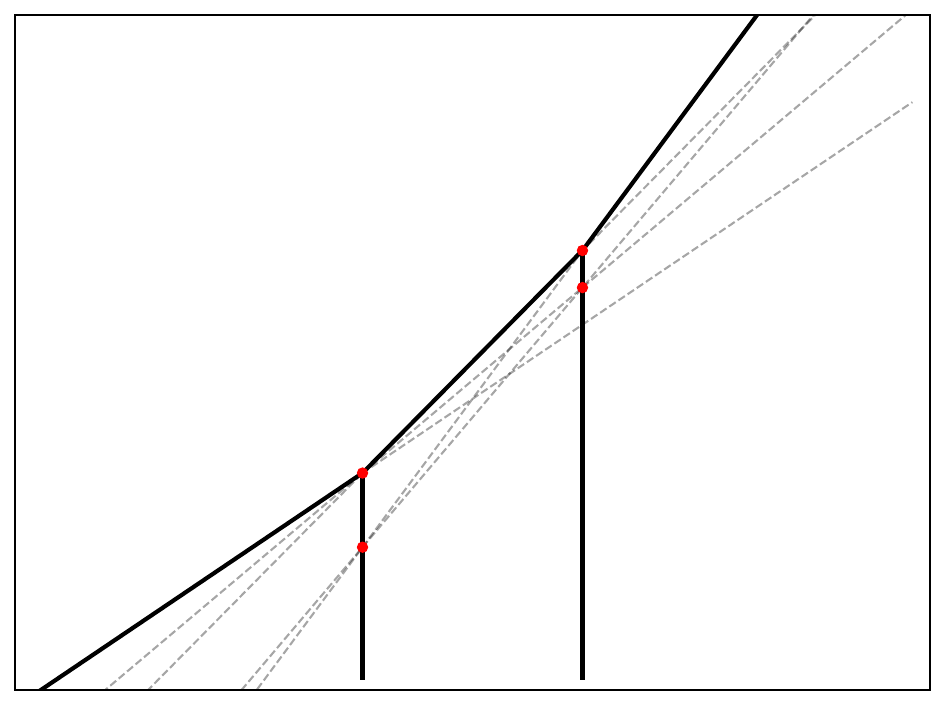}
        \caption{\microfont $0{\odot}x^{d+4} \oplus {-1}{\odot}x^{d+3} \oplus 2{\odot}x^{d+2} \oplus 1{\odot}x^{d+1} \oplus 0{\odot}x^d$}\label{plot_c}
    \end{subfigure}

    \vspace{0.7em} 
    \begin{subfigure}[b]{0.31\textwidth}
        \centering
        \includegraphics[width=\textwidth]{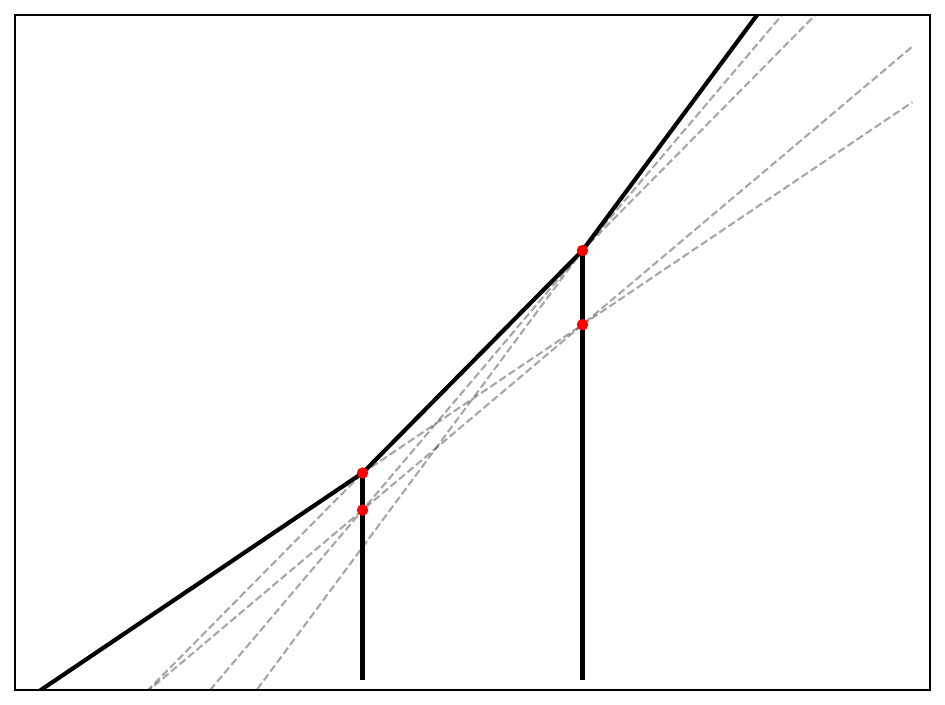}
        \caption{\microfont $0{\odot}x^{d+4} \oplus 1{\odot}x^{d+3} \oplus 2{\odot}x^{d+2} \oplus {-1}{\odot}x^{d+1} \oplus 0{\odot}x^d$}\label{plot_d}
    \end{subfigure}\hfill
    \begin{subfigure}[b]{0.31\textwidth}
        \centering
        \includegraphics[width=\textwidth]{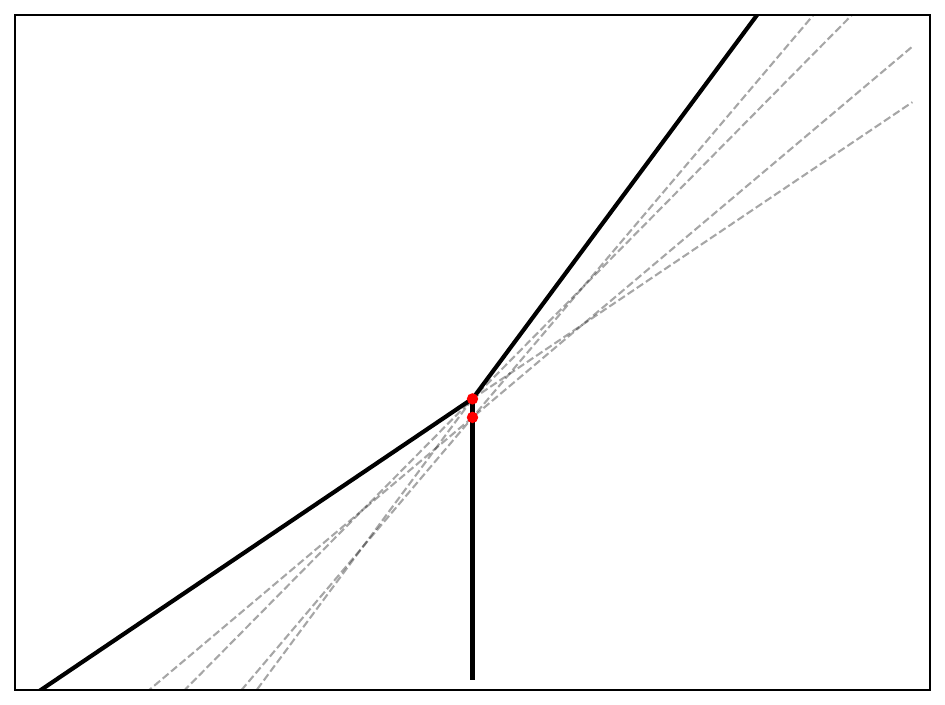}
        \caption{\microfont $0{\odot}x^{d+4} \oplus {-1}{\odot}x^{d+3} \oplus 0{\odot}x^{d+2} \oplus {-1}{\odot}x^{d+1} \oplus 0{\odot}x^d$}\label{plot_e}
    \end{subfigure}\hfill
    \begin{subfigure}[b]{0.31\textwidth}
        \centering
        \includegraphics[width=\textwidth]{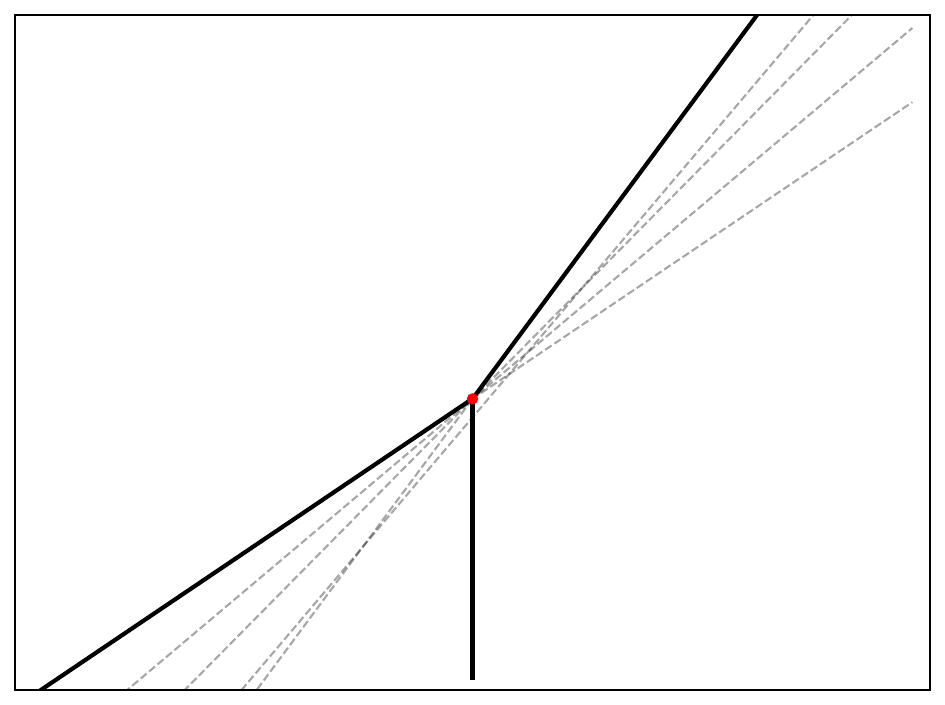}
        \caption{\microfont $0{\odot}x^{d+4} \oplus {-1}{\odot}x^{d+3} \oplus 0{\odot}x^{d+2} \oplus 0{\odot}x^{d+1} \oplus 0{\odot}x^d$}\label{plot_f}
    \end{subfigure}

    \vspace{0.7em} 
    \begin{subfigure}[b]{0.31\textwidth}
        \centering
        \includegraphics[width=\textwidth]{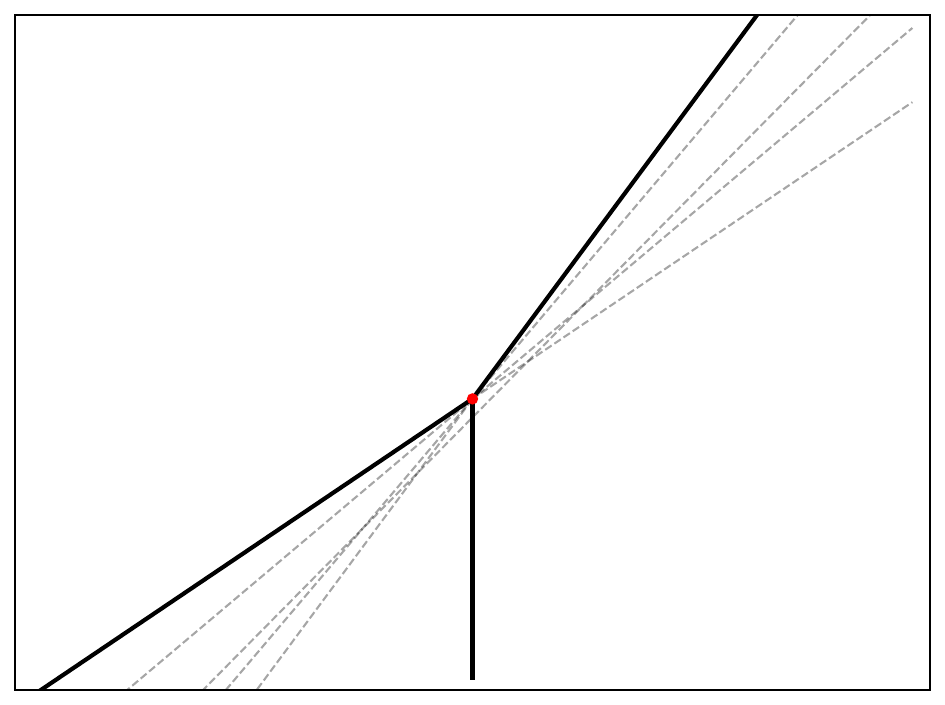}
        \caption{\microfont $0{\odot}x^{d+4} \oplus 0{\odot}x^{d+3} \oplus {-1}{\odot}x^{d+2} \oplus 0{\odot}x^{d+1} \oplus 0{\odot}x^d$}\label{plot_g}
    \end{subfigure}\hfill
    \begin{subfigure}[b]{0.31\textwidth}
        \centering
        \includegraphics[width=\textwidth]{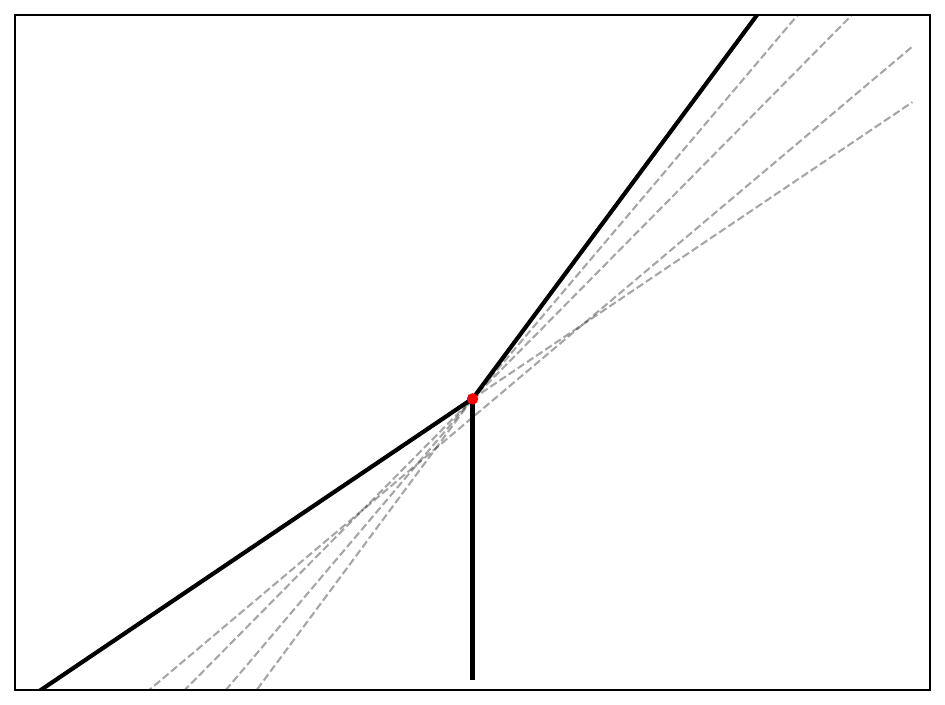}
        \caption{\microfont $0{\odot}x^{d+4} \oplus 0{\odot}x^{d+3} \oplus 0{\odot}x^{d+2} \oplus {-1}{\odot}x^{d+1} \oplus 0{\odot}x^d$}\label{plot_h}
    \end{subfigure}\hfill
    \begin{subfigure}[b]{0.31\textwidth}
        \centering
        \includegraphics[width=\textwidth]{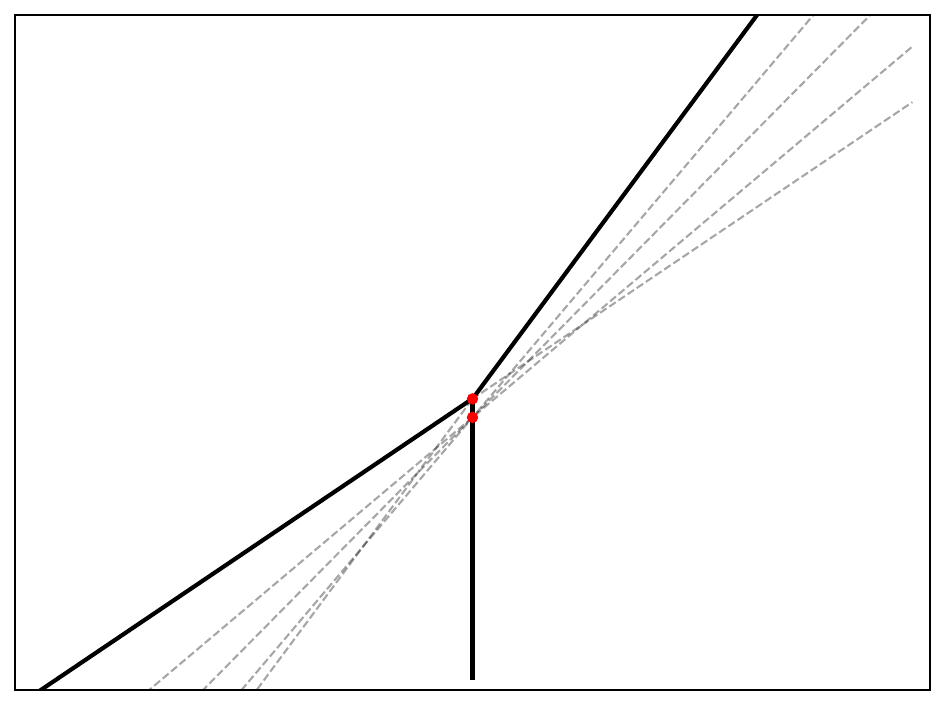}
        \caption{\microfont $0{\odot}x^{d+4} \oplus {-1}{\odot}x^{d+3} \oplus {-1}{\odot}x^{d+2} \oplus {-1}{\odot}x^{d+1} \oplus 0{\odot}x^d$}\label{plot_i}
    \end{subfigure}

    \caption{Plots of tropicalizations}
    \label{fig:tropical_3x3_polynomials}
\end{figure}

Now, we describe the distinction between $C_0$ and $C_2$ in terms of these plots. Cases \ref{plot_f}-\ref{plot_h} have equivalent structures of the intersections of the lines and occur for both $C_0$ and $C_2$. Cases \ref{plot_a}, \ref{plot_b} and \ref{plot_e} have the unique structure, but they occur for both $C_0$ and $C_2$. Case \ref{plot_i} has a unique structure, but occurs only for $d = 4k$. Thus, these seven cases do not provide a way to distinguish $C_0$ from $C_2$. The remaining two cases are more interesting:
cases \ref{plot_c} and \ref{plot_d} have equivalent structures of the intersections of the lines and occur only for $C_2$, namely, the intersection of three lines at two points in the plot and one triple intersection at a tropical root.

Based on these observations, we propose the following conjecture.

\begin{deff}
    Let $f$ be a tropical polynomial. A point is called \it{multiple} if at least three lines corresponding to the monomials of $f$ intersect at that point. We may construct a graph with vertices corresponding to the monomials of $f$: we connect two vertices by an edge if the intersection of the corresponding lines on the plot is multiple. We call this graph the \it{combinatorial type} of $f$. The combinatorial type is defined up to graph isomorphism.
\end{deff}

\begin{conjecture}
    \begin{enumerate}
        \item  Let $A$ be a support. Each topological equivalence class of polynomials in $\CC^A$ is uniquely defined by the set of combinatorial types of the tropicalizations of its branches at infinity. Call it \it{the set of combinatorial types} of a class.
        \item Consider the number of topological equivalence classes in $\CC^{A+d}$ as a function $N_A(d)$. We conjecture that $N_A(d)$ is periodic in $d$. 
        \item Let $p$ be the period of $N_A$. There is a bijection between the topological equivalence classes in $\CC^A$ and those in $\CC^{A+p}$ whose sets of combinatorial types coincide.
        \item Each topological equivalence class of polynomials in $\CC^A$ has such a combinatorial type, which occurs for at most one topological equivalence class of polynomials in $\CC^{A+d}$ for each $d\in \ZZ$. 
    \end{enumerate}  
\end{conjecture}

\end{section}


\begin{thebibliography}{99} 

\bibitem{AVG} V.~I.~Arnold, S.~M.~Gusein-Zade, and A.~N.~Varchenko, \emph{Singularities of Differentiable Maps. Vol.~I: The Classification of Critical Points, Caustics and Wave Fronts}, Moscow: Nauka, 1982, 304~pp.

\bibitem{BeBrM} B.~Bertrand, E.~Brugall\'e, and G.~Mikhalkin, ``Tropical Open Hurwitz numbers,'' \emph{Rendiconti del Seminario Matematico della Universit\`a di Padova} \textbf{125} (2011), 157--171. doi:10.4171/RSMUP/125-10. 

\bibitem{BS} E.~Brugalle , K.~Shaw , ``A bit of tropical geometry,''
\url{https://arxiv.org/abs/1311.2360} [math.AG] (2013).


\bibitem{CJM} R.~Cavalieri, P.~Johnson, and H.~Markwig, ``Tropical Hurwitz numbers,'' \emph{Journal of Algebraic Combinatorics} \textbf{32} (2010), 241--265. doi:10.1007/s10801-009-0213-0.

\bibitem{EV} A.~Esterov, A.~Voorhaar, ``Basecondary polytopes,'' \url{https://arxiv.org/abs/2411.02234} [math.AG, math.CO] (2024).

\bibitem{GKZ} I.~M.~Gel'fand, A.~V.~Zelevinsky, and M.~M.~Kapranov, ``Discriminants of polynomials in several variables and triangulations of Newton polyhedra,'' \emph{Algebra i Analiz} \textbf{2}, no.~3 (1990), 1--62; English transl.: \emph{Leningrad Mathematical Journal} \textbf{2}, no.~3 (1991), 499--505.

\bibitem{2} M.~Golubitsky and V.~Guillemin, \emph{Stable Mappings and Their Singularities}, Moscow: Mir, 1977.

\bibitem{4} I.~P.~Goulden and D.~M.~Jackson, ``The combinatorial relationship between trees, cacti and certain connection coefficients for the symmetric group,'' \emph{European Journal of Combinatorics} \textbf{13}, no.~5 (1992), 357--365. doi:10.1016/S0195-6698(05)80015-0.

\bibitem{GJV} I.~P.~Goulden, D.~M.~Jackson, and A.~Vainshtein, ``The number of ramified coverings of the sphere by the torus and surfaces of higher genera,'' \emph{Annals of Combinatorics} \textbf{4}, no.~1 (2000), 27--46. doi:10.1007/PL00001274.

\bibitem{KaKhE} B.~Ya.~Kazarnovskii, A.~G.~Khovanskii, and A.~I.~Esterov, ``Newton polytopes and tropical geometry,'' \emph{Russian Mathematical Surveys} \textbf{76}, no.~1 (2021), 91--175; original Russian version: \emph{Uspekhi Matematicheskikh Nauk} \textbf{76}, no.~1 (2021), 95--190. doi:10.1070/RM9937.

\bibitem{LZV} S.~K.~Lando and D.~Zvonkine, ``On multiplicities of the Lyashko--Looijenga mapping on strata of the discriminant,'' \emph{Functional Analysis and Its Applications} \textbf{33}, no.~3 (1999), 178--188. doi:10.1007/BF02465202.

\bibitem{LSH} E.~Shuvaeva, \emph{Multiplicity of the Lyashko--Looijenga Map for Laurent Polynomials}, Bachelor's thesis, HSE University, 2020. Available at: \url{https://www.hse.ru/en/edu/vkr/369849382}. 

\bibitem{3} D.~Trotman, ``Stratification theory,'' in J.~L.~Cisneros Molina, D.~Tr\'ang L\^e, and J.~Seade (eds.), \emph{Handbook of Geometry and Topology of Singularities I}, Springer, Cham, 2020, 243--273. doi:10.1007/978-3-030-53061-7\_4. 

\bibitem{Vo} A.~Voorhaar, ``The Newton Polytope of the Morse Discriminant of a Univariate Polynomial,'' \emph{Advances in Mathematics} 432 (2023), 109275.
doi:10.1016/j.aim.2023.109275.

\bibitem{ZD} S.~Zdravkovska, ``The topological classification of polynomial mappings,'' \emph{Uspekhi Matematicheskikh Nauk} \textbf{25}, no.~4(154) (1970), 179--180.

\bibitem{ZV} A.~Zvonkin, ``Towards Topological Classification of Univariate Complex Polynomials,'' in \emph{Formal Power Series and Algebraic Combinatorics: 12th International Conference, FPSAC'00}, Springer, 2000, 76--87. \end{thebibliography}
\end{document}